\documentclass[11pt]{amsart}
\usepackage{geometry}
\usepackage{bbm}               % See geometry.pdf to learn the layout options. There are lots.
\usepackage{graphicx}
\usepackage{amssymb}
\usepackage{epstopdf}
\usepackage{color}
 \usepackage{hyperref}

\DeclareMathAlphabet{\mathpzc}{OT1}{pzc}{m}{it}

\def\N{\mathbb N}

\def\R{\mathbb R}

\newtheorem{theorem}{Theorem}
\newtheorem{proposition}{Proposition}[section]
\newtheorem{lemma}{Lemma}

\newtheorem{remark}{Remark}

\newcommand{\vare}{\varepsilon}

\title[Transfer of energy from a flexural to a torsional mode]{Transfer of energy from   flexural to   torsional modes for the Fish-bone Suspension bridge}
\author {Clelia  Marchionna}  \address{Dipartimento di
Matematica del Politecnico, Piazza Leonardo da Vinci 32, 20133
Milano, ITALY}
\author{Stefano Panizzi} \address{Dipartimento di Scienze
Matematiche, Fisiche  e Informatiche, Parco Area delle Scienze 53/A, 43126 Parma, ITALY}

\begin{document}
%\maketitle
\begin{abstract} We consider a conservative  coupled oscillators systems  which arises as a   simplified model of the  interaction of flexural and torsional modes of vibration along the deck  of the so-called fish-bone \cite{BG} model of suspension bridges.
The elastic response of the cables is supposed to be asymptotically linear under traction, and asymptotically constant when compressed (a generalization of the slackening regime).  We show that for vibrations of sufficiently large amplitude, transfer of energy  from  flexural modes  to torsional modes may occur provided a certain condition on the parameters is satisfied.
The main result is a non-trivial extension of a theorem in \cite{MP} to the case when the frequencies of the normal modes are no more supposed to be the same. Several  numerical computations of   instability diagrams for    various  slackening models respecting our assumptions are presented.  \end{abstract}

\maketitle

 \vspace{0.2cm} \\

\section{Introduction}
%An important issue in the mathematical modeling is the energy transfer between different modes of vibration in absence of external forces: the so called internal resonance
In this  paper, which  is a completion of a previous work \cite{MP}, we face the question of energy  exchange  between  torsional and longitudinal modes in a generalized
energy conserving suspension bridge model, the so-called \emph{fish-bone} model, proposed by K.S. MOORE  \cite{Moo},
in which  the  elastic response of the cables is   asymptotically linear under traction, and asymptotically constant when compressed (a generalization of the slackening regime).

We follow the line of research carried out by F. GAZZOLA and coworkers in a series of papers \cite{AG,BG, Gaz} according to which  internal nonlinear resonances   may occur even when the aeroelastic coupling is disregarded.
To analyze the onset of resonances, the PDEs system (infinite dimensional system) is reduced to a coupled system  of nonlinear ODEs,
 by projecting the infinite dimensional phase space on a two dimensional subspace through an approximate  Galerkin technique.
The ODEs system is then a simplified model of the  interaction of flexural and torsional modes of vibration along the deck of the bridge.

%:
The same problem was addressed in a previous paper \cite{MP}, in which only the interaction between the modes corresponding to the fundamental flexural and torsional frequencies was considered.
In the present work, we extend the study to the more general case in which modes of different frequencies can interact.

The main result is  Theorem \ref{main} in which we prove  that for deflections  of sufficiently large amplitude, pure flexural and periodic vibrations are unstable,  provided a certain condition on the structural parameters of the bridge is satisfied, so that transfer of energy  from  flexural modes  to torsional modes may occur
even when no external forces are applied.  

% It is worth noting that,   the instability condition can only be satisfied by odd $j$....

%This result  is a non-trivial  extension to the case when the frequencies of the normal modes are no more supposed to be the same, of the instability condition of pure flexural motion at large energy   established in  \cite{MP}.  

 Our technique follows the lines of previous literature:  we consider the differential of the Poincaré map relative 
to the torsional component of the system, around a periodic purely flexural solution. As is well-known this leads to   a Hill equation; then,  as  the energy of the system tends to infinity,  we compute the limit  equation, which turns out to be simple enough  to allow the analysis of its stability via the classical Floquet theory. 

If the transverse and torsional modes
are labeled by $j$ and $k$ respectively, an unexpected property holds when $j$ is even: the instability condition at large deflections of Theorem \ref{main}     is never satisfied, as the limit system is an  uncoupled  system of linear, constant coefficients oscillators. 
Consequently,   every solution of the limit system is periodic  thus stable, so that    purely flexural solutions are asymptotically linearly stable. 

% This obviously does not mean that there are not unstable pure flexural solution for  (\ref{ODE}) even for large $q$, but we can expect a narrowing of instability regions as
% $\, q \rightarrow \infty$.

The extension of  the result in \cite{MP} is non-trivial because  2 main steps are involved: the quite subtle derivation of the $C^1$-regularity of the projected system under the relaxed regularity assumption ($\bf S_0$) on the slackening function; the computation of the limit Hill equation  at large energies, which requires several technicalities.

In order to obtain a more general picture of the stability of the flexural component, in Section 3 we  look at  the Hill equation relative  to the torsional component as dependent on two parameters. The first natural parameter is the maximum elongation of the purely flexural solution, which in Theorem 1 tends to  $ \infty$; we note that the dependence on such parameter is nonlinear. The choice of the second parameter is motivated by  mathematical  reasons, as the equation naturally presents a spectral parameter which, from the structural point of view, essentially corresponds to the torsional behavior of the bridge.  In this way, we can compare the results related to our problem with the instability diagrams  in the literature   for  classical two-parameters Hill equations. 
In particular,  we consider a slackening model that has minimal regularity with respect to our requests, using both academic and real parameters corresponding to Tacoma narrow bridge; we numerically draw  the corresponding    instability diagrams, pointing out which properties are preserved and which diverge from the classical ones; some interesting mathematical aspects are highlighted, such as the presence of resonance pockets typical of some two parameters Hill equations, {\em e.g.} the multi-step Meissner equation.

%As is common in the study Hill equations 
%A common procedure  for studying in detail the instability of the null solution of a Hill Equation is to introduce a pair of parameters, and discuss the instability in the plan of the parameters by studying the so-called instability tongues. For our problem, the first natural parameter is the maximum elongation of the purely flexural solution, which in Theorem 1 tends to + \ infty;  

%A common technique for studying in detail the instability of the null solution of a Hill Equation is to introduce a pair of parameters, and discuss the instability in the plan of the parameters by studying the so-called instability tongues. For our problem,  To compare the results related to our problem with those present in the literature, from a mathematical point of view it is natural to choose the spectral parameter of the Hill equation as the second parameter.

%In Section 3 we present the numerical computation of  several instability diagrams  of   various  slackening models respecting our assumptions. We use both academic and real parameters corresponding to  Tacoma  narrow bridge. Some interesting mathematical aspects are highlighted, such as the presence of resonance  pockets typical of some two parameters Hill equations, {\em e.g.} the multi-step Meissner equation.
The remainder of this section is devoted to the presentation of the PDEs model,  and  its  reduction to the ODEs system. In Section 2 we provide the proof of the main instability result. 
The work ends with three appendices: in Appendix A we prove the regularity of the ODEs system;  the computation  of the limit system at large energies is presented in
Appendix B; finally, Appendix C provides the explicit computation of the instability discriminant for a piecewise linear system leading to a
multi-step Meissner equation.

\subsection{The bridge model}
 We briefly recall the PDEs system. The dynamics   of the midline of the deck, modeled as an  Euler-Bernoulli  beam of length $L$ and width $2l$, is
 coupled with the  elastic response of the suspension cables acting  on the side ends of the deck.
 The  cross section of the deck is assumed to be a rigid rod   with  mass density $\rho$, length $2l$ and  thickness small   with respect  to $l$;
$Y(x,t)$  is the vertical downward deflection  of the midline of the deck with respect to the unloaded state,
$\Theta(x,t)$  is the angle of rotation of the deck with respect to the horizontal position.
The corresponding  PDEs system is given by,
\begin{equation}
\label {PDE1}
\begin{cases}
Y_{tt} + \frac{EI}{\rho S} \,Y_{xxxx} + f(Y+l\sin \Theta) +  f(Y -l\sin \Theta) =0, \\[1ex]
    \Theta_{tt} - \frac{G K }{\rho J} \Theta_{xx} + \, \frac{Sl}{J} \cos \Theta \left[f(Y+l\sin \Theta) -f(Y -l\sin \Theta) \right] =0,
\end{cases}
\end{equation}
complemented with hinged boundary conditions:
\begin{equation}
\label{hinged} Y(0,t)=Y(L,t) = Y_{xx}(0,t)= Y_{xx}(L,t)=0, \quad \Theta(0,t) = \Theta(L,t)=0.
\end{equation}

The other constant parameters are:  $S$  the cross section area, $I$  the
planar second moment of area with respect to the plane $Y=0$, $J$  the polar second moment of area with respect to the $x$-axis  and $E$ and $G$  respectively the  Young  modulus and the shear modulus, K the torsional constant.

The restoring force  $f$ exerted by the hangers   is  applied to both  extremities of the deck  whose displacements from the unloaded state are given by $Y \pm l\sin \Theta$.  No external forces, except gravity, are taken in account.

 In the classical slackening regime,  the hangers  behave as   linear springs of elastic constant ${\rm k}>0$  if stretched and do not exert restoring force if compressed.
A first model in which the system (\ref{PDE1}) acts as a linear, non coupled system for small displacements,  was proposed by K.S MOORE and P.J. McKENNA  \cite{MCKW,Moo}
(hereforth MMK model) assuming for  $f$ the following expression ($g$ is the gravity),
\begin{equation}
\label{slack}
f(r) = m  \left[ (r+r_0)^+ - r_0 \right], \quad r_0 = \rho Sg/2{\rm k}, \quad m = \frac{{\rm k}}{\rho S}.
\end{equation}

Subsequently many other forms  for $f$  have been proposed
 in \cite{BG, MP,  MP2, MCKT}, some of these are nonlinear and smooth  in a neighborhood of the origin, making instability feasible even at  low energies.
 Two  significant examples
 are ($h$ is a positive constant):
 \begin{equation}
\label{fradice}
f(r)=m r +\sqrt{(mr)^2+h^2}\, \,-h,
\end{equation}
\begin{equation}
\label{fexp}
f(r)=h(e^{mr/h}-1),
\end{equation}

 Throughout the paper  we  assume that the  function $f$ satisfies  the following  mild regularity condition:
  \begin{center} \emph{ Assumption} ($\bf S_0$) \end{center}
\begin{itemize}
  \item[a)] $f$\emph{ is  continuous, strictly increasing, and    $f(0)=0$};
  \item[b)] $f$ \emph{ is piecewise $C^1$}, that is \emph{ its derivative is continuous with the exception of a finite
  (eventually empty) set  of points $r_1 < r_2 <...<r_n$ {\textbf { not including zero}} in which   there exist the finite limits: }
  \begin{equation*}
   \lim_{r\rightarrow r_i^{\pm}}f^\prime(r);
   \end{equation*}
 \item[c)] $m:= f^\prime(0)>0$.
\end{itemize}

 In most cases the elastic response of the cables is supposed to be asymptotically linear under traction, and asymptotically constant when compressed, thus it is natural to assume at least one of  the following conditions:
\begin{itemize}
  \item[($ \bf S_1$)]  $\lim_{r\rightarrow -\infty}f^{\prime}(r) = 0 $
  \item[($ \bf S_2$)]  $ M:= \lim_{x\rightarrow +\infty}f^{\prime}(r) >0.$
  \end{itemize}

We note that both  (\ref{slack}) and  (\ref{fradice})  satisfy $(\bf S_0)$ - $(\bf S_1)$ - $(\bf S_2)$,
unlike the function  (\ref{fexp}) which does not satisfies condition $(\bf S_2)$.

The  problem  (\ref{PDE1})-(\ref{hinged})  is well-posed in  appropriate Sobolev spaces  \cite{BG,F}, and enjoys
 two  properties mostly relevant for our purposes:  the total energy is conserved over time; it admits  \emph{pure flexural solutions}, that is motions in which the cross sections of the deck remain horizontal at all times  so that  no torsional vibrations   occur.

\subsection{The ODEs system}

For  small torsional angles, it is very convenient to  replace the system (\ref{PDE1}) with a   pre-linearized one, see \cite{BG,MP}.
By the  usual approximation: $\sin \Theta \sim \Theta$, $\cos \Theta \sim 1$, and by setting $Z=l\Theta$, the system (\ref{PDE1}) reduces to \begin{equation}
\label {PDE2}
\begin{cases}
 \,Y_{tt} + \frac{EI}{ \rho S} \,Y_{xxxx} +f(Y+Z) +  f(Y-Z) =0 \\[1ex]
  \,  Z_{tt} - \frac{G K}{ \rho J} \, Z_{xx} + \, \frac{l^2S}{J} \left[f(Y+Z) -f(Y -Z) \right] =0,
\end{cases}
\end{equation}
and as far as  the scope of this paper is concerned, nothing changes  starting from the system (\ref{PDE1}) or  from  (\ref{PDE2}).

Our \emph{ansatz} is that, after a suitable rescaling of the space variable, the displacements can be  reasonably well approximated by the $j-k$ mode of vibration, that is,
\begin{equation*}
Y(x,t) \simeq y_j(t) \sin (jx), \quad Z(x,t) \simeq z_k (t) \sin (k x),\quad 0 \leq x \leq \pi.
\end{equation*}

Then, through a Galerkin projection,  the PDEs system reduces to a coupled oscillators system which,  dropping  the indexes $j-k$, reads as follows:
\begin{equation}\label{ODE}
\begin{cases}
\ddot{y}   + \alpha j^4 y  +\psi_1(y,z)  = 0 \\[1ex]
\ddot{z}   +  \beta k^2 z   + \gamma\psi_2(y,z)    = 0,
\end{cases}
\end{equation}
with structural parameters, $\alpha=  EI\pi^4/ \rho S L^4$, $\beta= GK \pi^2/ \rho JL^2$, $\gamma = l^2S/J$, and   nonlinear coupling terms,
\begin{align}
&\psi_1(y,z) =
\frac{2}{\pi} \int_0^{\pi} \left[f(y\sin(j x)+z\sin(kx)) + f(y\sin(j x)-z \sin (kx))\right] \sin (j x) \, dx, \label{psi1} \\
&\psi_2(y,z) =
 \frac{2}{\pi} \int_0^{\pi} \, \left[f(y\sin (j x)+z\sin(kx)) - f(y\sin(j x)-z \sin(kx))\right] \sin(kx) \, dx. \label{psi2}
\end{align}

If we   define
 \begin{equation*}
 \Psi(y,z) \, = \, \frac{2}{\pi} \int_0^{\pi} \left [ \mathcal{F}(y\sin(j x)+z \sin(k x)) + \mathcal{F}(y\sin(j x)-z \sin(k x))\right]  \, dx,
  \end{equation*}
 where $\mathcal{F}(r) = \int_0^r f(s) \, ds$, we have
  $  \psi_1  \, = \,  \partial \Psi/\partial y $, $ \psi_2  \, = \,   \partial \Psi/\partial z $, so that
the system  (\ref{ODE})    admits
 a conserved   energy,
  \begin{equation}
\label{energia}
  \mathcal{E}(y,\dot{y},z, \dot{z}) \, = \, \frac{\dot{y}^2}{2} + \frac{\dot{z}^2}{2 \gamma}  + \frac{\alpha j^4}{2} y^2 + \frac{\beta k^2}{2\gamma}z^2 + \Psi(y,z).
  \end{equation}

Note that, under Assumption  ($\bf S_0$), $\Psi(y,z)$ is  nonnegative. As a consequence all solutions of (\ref{ODE}) are global and bounded.

Since $\psi_2 (y,0) \equiv 0$, the system (\ref{ODE}) admits periodic pure flexural solutions, that is  solutions of the form $y=u(t)$, $z\equiv0$ with $u(t)$  periodic.
We consider such solutions as parametrized  by the initial displacement, and we define  $u= u(t;q)$ as the solution of  the initial value problem,
\begin{equation}
\label{eqq}
\ddot{u}  + \alpha j^4 u  + 2 f_j(u ) = 0 \qquad u(0)=q, \quad \dot u (0) =0.
\end{equation}
where
\begin{equation}
\label{ftildej}
f_j(r) :=\frac{1}{2}\, \psi_1(r,0)=\frac{2}{\pi} \int_0^{\pi}f(r \sin jx) \sin jx \, dx.
\end{equation}

Assuming  for the moment that $f\in C^1$, the linearization at a fixed energy level (iso-energetic linearization) of the system  around the periodic orbit $(u(\cdot, q), 0)$ yields, for the torsional component, the Hill equation  (see {\em e.g.} \cite{CW,MP} for details),
\begin{equation}
\label{zjk}
\ddot{v}   +\left( \beta k^2   +  2  \gamma  g_{j,k}(u(t;q))  \right )v    = 0,
\end{equation}
in which we have set,
\begin{equation}
\label{gjk}
g_{j,k}(r) =\frac{1}{2}\,  \frac{\partial  \psi_2}{\partial z}(r,0)=\frac{2}{\pi} \int_0^{\pi}f^\prime(r \sin jx) \sin^2 kx \ dx.
\end{equation}

The problem we want to address is the stability of   solutions of  the Hill equation (\ref{zjk}).
It is worth noting that  in the case  $j=k$, we have $g_{j,j}(r)=f^\prime_j(r)$, and   the linearized system (\ref{eqq})-(\ref{zjk})    is the same as the one studied  in \cite{MP}. This is no longer true if $j\neq k$.
In \cite{MP}, under the assumptions $\bf(S_0)$ - $\bf(S_1)$ - $\bf(S_2)$, we established
 a condition depending on a set of 3 parameters under which the flexural motions are unstable provided the energy  parameter $q$ is sufficiently large.
The next section is devoted to the main result of the present paper which is a non-trivial  extension of the result in \cite{MP} to the case $j\neq k$.

\section{Instability of pure flexural $j-k$ modes at high energies}
\label{high energies}

The stability analysis of (\ref{zjk}) is carried out by means of Floquet's theorem, see \cite{E, MW}.
We  recall the  definition of the stability discriminant  $\Delta= \Delta(q)$ of the Hill equation.
Let $\tau(q)$ be the period of the solution $u$ of the problem (\ref{eqq}), and let    $v_0(t)$,  $v_1(t)$ be  the solutions of (\ref{zjk}) corresponding to initial data
$$v_0(0)=1,\, \dot v_0(0)=0,\quad v_1(0)=0,\, \dot v_1(0)=1;$$
then
%\begin {equation}
%\label{Delta}
$$
\Delta(q)=v_0(\tau (q))+v'_1(\tau (q)).
$$
%\end{equation}

 If  $|\Delta|>2$ the non-trivial solutions of the Hill  equation are unbounded, if  $|\Delta|<2$ are all bounded. In the case
when $\Delta=2$ there exists at least a non trivial $\tau (q)$-periodic solution, when  $\Delta=- 2$ there exists at least a non-trivial $2\tau (q)$-periodic solution.

The main result of this section consists in the  the computation  of
$$ \Delta_{\infty} :=  \lim_{q \rightarrow \infty} \Delta (q), $$
in the case when the elastic response of the cables is asymptotically linear, {\em i.e} under assumptions $\bf(S_1)$ - $\bf(S_2)$.
Referring to the characterization of instability $|\Delta_\infty|>2$, we can establish  a condition, depending on $j$, $k$, and the structural parameters, for which there is instability for sufficiently high energies.
The proof    mimics that of   Theorem 1 in \cite{MP}.
First  we are able to compute the limit system of (\ref{eqq})-(\ref{zjk})
 as  $q$ goes to $+\infty$; it turns out  that the  Hill  equation of the limit system is  a  two-step potential (Meissner equation) that  can be integrated explicitly;  the condition (\ref{c1}) expresses the instability condition for $q$ sufficiently large through the discriminant $\Delta_\infty$ of the limit equation.

%An unexpected property holds when $j$ is even: the limit system  is an  uncoupled  system of linear, constant coefficients oscillators.
%Then every solution of the limit system is periodic  thus stable.
% This obviously does not mean that there are not unstable pure flexural solution for  (\ref{ODE}) even for large $q$, but we can expect a narrowing of instability regions as
% $\, q \rightarrow \infty$.
 %The question will be taken up again in the numerical simulations in Section 3.

%%%%%%%%%%%%%%  TEOREMA 1 %%%%%%%%%%%%%%%%%%%%%%%%
\begin{theorem}
\label{main}
Assume that the function  $f$ satisfies the conditions  $\bf(S_0)$ - $\bf(S_1)$ - $\bf(S_2)$.

Assume that    $\,j \, $ is odd, and let the constants  $\omega_{\pm}$,   $A_{\pm}$,   $\phi_{\pm}$,   be defined as follows:
$$
\omega_{\pm}^2 = \alpha j^4 +M \, \left(1 \pm \frac{1}{j}\right),
$$
%\begin{equation}
%\label{ABdispari}
$$
A_{\pm}^2=\beta k^2 +\gamma M(1\pm \epsilon_{j,k}) ,  \quad  \epsilon_{j,k} = \frac{1}{j}  - \frac{\tan( \pi k /j )}{k\pi},
$$
%\begin{equation}
%\label{fi}
$$
\phi_{\pm}   =\frac{A_{\pm}  }{\omega_{\pm} }  \,  \pi ,  \qquad a = \frac{A_+}{A_-}.
$$

Then, if  the following condition holds true,
\begin{equation}
\label{c1}
\frac{|\Delta_\infty|}{2}:=\left| \cos \phi_+ \,  \cos \phi_-- \frac{ a + a^{-1} }{2} \, \sin \phi_+ \, \sin \phi_- \,  \right| >1.
\end{equation}
there exists $\, q_0 \,$ such that, if $ \, q>q_0$,   the pure flexural periodic solution $(u(t;q),0)$ of the  (non-linear) system (\ref{ODE})  is unstable.
\end{theorem}

From the above expression of $A_{\pm}^2$  it is not obvious that they are positive.  Actually the two quantities   $1 \pm \epsilon_{j,k} $ result as the two possible values ($+$ for $r>0$, and $-$ for $r<0$) of the  positive integral
$ \frac{4}{\pi} \int_0^{\pi} H(r \sin jx) \sin^2 kx \ dx$, see (\ref{AB}), (\ref{jodd}) below, and Appendix B.

\begin{remark}
In   Theorem \ref{main} we have considered only the case of odd $j$   as for  even $j$  the condition (\ref{c1}) is never satisfied. Indeed, thanks to (\ref{jeven})  the limit system is  simply  a decoupled linear system that  we can write explicitly as follows,
\begin{equation*}
\begin{cases}
\ddot{U}_\infty +  (\alpha  j^4 +M)U_\infty  = 0 \\[1ex]
\ddot{v}_\infty   +( \beta k^2 +  \gamma M) v_\infty    = 0.
\end{cases}
\end{equation*}

\end{remark}

\medskip
{\subsection{ Linearized $ j-k$ mode: technical tools}
\label{proof teorema ponte}
First of all we need a  regularity result which is crucial  for  the linearization process  of the system (\ref{ODE}) around a pure flexural solution.
\begin{proposition}
\label{psi}
If the function $f$ satisfies assumption ($\bf{S_0}$), then the functions   $\psi_1$, $\psi_2 $, as defined  in (\ref{psi1})-(\ref{psi2}), are of class $  C^1(\mathbb{R}^2)$.
\end{proposition}
The proof of this Proposition  is given in Appendix A.

Then we  list  some general facts about the functions $ f_j(r)$, $g_{j,k}(r)$ defined in (\ref{ftildej}),(\ref{gjk}).
Given any function $\, f \,$, accordingly with \cite{MP}, we define the  integral transform,
\begin{equation}
\label{defftilde}
\tilde{f}(r) := \frac{2}{\pi} \int_0^{\pi} f(r \sin x) \sin x \, dx.
\end{equation}

We denote by $\, f_e(x) = \frac{1}{2} \left(f(x) + f(-x)\right) $, $\, f_o(x) = \frac{1}{2} \left(f(x) -  f(-x)\right) $  even and odd parts of a function respectively.
We have the following,
\begin{lemma}
\label{lemmaftilde}
 If the function $\,f\,$ satisfies the first two conditions of the assumption ($\bf S_0 $), then $ f_j\in C^1(\mathbb{R})$, $f_j(0)=0$ and $ f'_j \geq 0$.

 Moreover we have,
 \begin{equation}
\label{ftildaj}
 f_j = \tilde{f_o} \quad (\text{even } j), \qquad  f_j  =  \tilde{f_o} + \frac{1}{j}\, \tilde{f_e} \quad (\text{odd } j ).
 \end{equation}
\end{lemma}

\begin{proof}
The regularity and the sign of  $ \,f'_j\, $ follow  from the analogous  properties of  $ \tilde{f}$ proved in  \cite{MP}, condition  $\mathbf{(\tilde{H})}$.
To  prove (\ref{ftildaj}), we use the elementary identity  $\, \sin(z + h\pi) = (-1)^h \sin z $, to compute
\begin{eqnarray*}
\int_0^{\pi} f(r \sin jx) \sin jx \,dx & = & \frac{1}{j}\int_0^{j\pi} f(r \sin z) \sin z \,dz \\
 & = & \frac{1}{j} \sum_{h=0}^{j-1} (-1)^h \int_0^{\pi} f((-1)^h r\sin z ) \sin z \, dz.
\end{eqnarray*}

By collecting the signs, and by distinguishing the two cases  even or odd $j$, we obtain (\ref{ftildaj}).
\end{proof}

\begin{lemma}
\label{j pari}
If $j$ is even, then    $g_{j,k}(r)$ is an even function for every $k$.  More precisely, we have
%Moreover, only the odd part of $f$ acts. That is, being $f_o(r)=\frac{1}{2}(f(r)-f(-r))$,
%$$f_j(r) = \frac{2}{\pi} \int_0^{\pi} f_o(r \sin x) \sin x \, dx,$$
%\begin{equation}
%\label{gtilde}
$$
 g_{j,k}(r)=\frac{2}{\pi} \int_0^{\pi} f_o'(r \sin jx) \sin^2 kx \, dx.
 $$
%\end{equation}
\end{lemma}

\begin{proof}
If $j$ is even, $f'(r \sin jx) (\sin kx)^2$ has period $\pi$, then

$$ g_{j,k}(r)=\frac{2}{\pi} \int_0^{\pi} f'(r \sin jx) \sin^2 kx \, dx=\frac{2}{\pi} \int_{-\pi/2}^{\pi/2} f'(r \sin jx) \sin^2 kx \, dx.$$
%It is easy to show, with the change of  the integration variable $x'=-x$, that $g_{j,k}(-r)=g_{j,k}(r)$.
It follows that
$$ g_{j,k}(r)=\frac{2}{\pi}\int_{-\pi/2}^{\pi/2} (f_e'(r \sin jx) \sin^2 kx +f_o'(r \sin jx) \sin^2 kx \, )\,dx $$
and the integral of the first term is null since $f_e'$ is odd.
\end{proof}

We make an observation to highlight some differences that may exist between even and odd harmonics of the flexural component.
If  the function $f$ is as in example  (\ref{fradice}), its odd part is linear, precisely we have
$f_o(r)=mr$. Then, if $j$ is even, the equation (\ref{zjk}) reduces to the linear oscillator,
 $$\ddot{ v}     +  ( \beta k^2 + 2  \gamma m) v=0.$$

Since  $\beta k^2 + 2  \gamma m>0$, the $j-k$ modes are always linearly stable. This is no more true for  odd $j$, see formula (\ref{ftildaj}).

%%%%%.

We end this section by introducing the functions which characterize the limit system as $q\rightarrow \infty$   ($(\cdot)^+ = $ positive part, $H(\cdot)$ Heaviside function). We set
\begin{equation}
\label{AB}  h_j(r) =\frac{2M}{\pi} \int_0^{\pi} (r \sin jx)^+ \sin jx \ dx, \qquad s_{j,k}(r) =\frac{2M}{\pi} \int_0^{\pi} H(r \sin jx) \sin^2 kx \ dx,
\end{equation}
then, by setting $r_0=0$  and   $\theta=0$ in Lemma \ref{teorematheta} in Appendix B, we obtain their  explicit expressions:
%\begin{lemma}
%\label{contilim}
%We have
\begin{equation}
\label{jeven}
h_j(r) =\frac{M}{2} r, \quad  s_{j,k}(r)= \frac{M}{2}  \qquad    (\text{even } j);
\end{equation}
\begin{equation}
\label{jodd}
h_j(r) =\frac{M}{2}\hspace{-0.5ex}\left[1+\frac{{\rm sign} (r)}{j} \right] r, \quad  s_{j,k}(r)= \frac{M}{2} \hspace{-0.5ex}\left [1+{\rm sign}(r)\hspace{-0.5ex} \left[\frac{1}{j} -
 \frac{\tan(k\pi/j)}{k\pi} \right]\right] \quad ( \text{odd } j).
\end{equation}

%%%%%%%%%
%
%
\subsection{Proof of Theorem \ref{main}}
\begin{proof}
In   \cite{MP} we proved the same facts for the simple mode with $j=k=1$, therefore here we just outline the main points emphasizing  a few differences.

Let $u$  be the solution  of the problem (\ref{eqq}), and $v$ be  the solution equation (\ref{zjk}) with fixed initial data
$  v (0)=a,\,\dot {v} (0)=b$.

We rescale $u$ by setting $U_q(t)=u (t)/q$,  then the system  becomes
\begin{equation}
\label{sistema_hat2}
\begin{cases}
\ddot{U}_q  + \alpha j^4 U_q +2 \displaystyle{\frac{f_j(qU_q )}{q}} = 0 \\[1ex]
\ddot{v}_q   +(k^2 \beta+2\gamma g_{j,k}(qU_q(t) ) ) v_q    = 0
\end{cases}
\end{equation}
with initial data
$U_q(0)=1,\,\dot{U}_q(0)=0$, $ v_q(0)=a,\,\dot {v}_q(0)=b$.

Then we introduce   the limit system  of (\ref{sistema_hat2}), as $q \rightarrow \infty$,
\begin{equation}
\label{sistema_lim}
\begin{cases}
\ddot{U}_\infty+  \alpha  j^4 U_\infty +2 \displaystyle{h_j(U_\infty)} = 0 \\[1ex]
\ddot{v}_\infty  +( \beta k^2 +2 \gamma  s_{j,k}(U_\infty(t)) v_\infty   = 0
\end{cases}
\end{equation}
with initial data
$ U_\infty(0)=1,\,\dot{U}_\infty(0)=0$, $ v_\infty(0)=a,\,\dot v_\infty(0)=b$.

We claim that, as $q\to +\infty$, the solutions of (\ref{sistema_hat2})-(\ref{sistema_lim}) satisfy the limit relations:
\begin{itemize}
\item[i)] $U_q \to U_\infty \quad $  in $C^0(\R)$;
\item[ii)] $v_q  \to v_\infty \quad $ in $C^1([0,T])$, for every $T>0$;
\item[iii)] $\tau_q \to \tau_\infty, \quad $  where $\tau_q$, $\tau_\infty$ are the periods of $U_q$,   $U_\infty$ respectively.
\end{itemize}
\noindent

Proof of i): Let $\tilde f(r)$  be as in (\ref{defftilde}), then in  \cite{MP}  we proved that
$\lim_{q\to+\infty}\tilde f(qr)/q =\tilde h(r)$,
uniformly on compact sets.
Thanks to   (\ref{ftildaj}), in the same way we obtain
$ \lim_{q\to+\infty} f_j(qr)/q = h_j(r)$, uniformly on compact sets.
Owing to a classical continuous dependence theorem (\cite{HS}, Th. 3, Ch.XV, p.297), we get that $U_q$  converges uniformly on $\R$ to $U_\infty$.

Proof of  ii):  Let us consider  the Hill  equation in (\ref{sistema_lim}). By Lemma 5.1 in \cite{MP} (see also \cite{O},)  we have to prove that $ g_{j,k}(qU_q(t))\to  s_{j,k} (U_\infty(t))$ in $L^1[0,T]$.

In the case when  $\sin jx\neq 0$, we have $U_\infty(t)\neq 0$, thus by the uniform convergence of $U_q$, for $q $ sufficiently large, the sign
of  $\,qU_q \sin jx \,$ is   the same as that of $\,U_\infty \sin jx$. Thanks to  the assumptions ($ {\bf S_1}$)-($ {\bf S_2}$), it follows that
$$\lim_{q\rightarrow + \infty} f'(qU_q(t)\sin jx)= M H(U_\infty(t)\sin jx).$$

Since $U_\infty(t)\neq 0$ almost everywhere in $[0,T]$, and  $s_{j,k}$ is bounded,  we get the required convergence of $ g_{j,k}(qU_q )$ to $ s_{j,k}(U_\infty )$.

 The  proof  of iii)  is the same as in Proposition 5.2 of  \cite{MP}.

Finally, we come to the computation of the limit period $\, \tau_{\infty} \,$ and discriminant $\, \Delta_\infty$.
Recall that $j$ is odd,  then   the limit system  is a coupled (nonlinear) system with step coefficients $h_j(r)$, $ s_{j,k}(r)$ defined by formula (\ref{jodd}).
It turns out that the limit Hill equation  (\ref{sistema_lim})  is a two-step  Meissner equation  that can be integrated explicitly.
First of all we note that the function $U_\infty(t)$ is even, so that we can fix our attention only on the half period. If we divide the interval $[0,\tau/2]$ in two subintervals $I^+=[0,t_0]$, where $U_\infty\geq 0$, and $I^-=[t_0,t_1]$, where  $U_\infty  \leq 0$, we easily obtain
$$t_0=\frac{\pi}{2\omega_+},\qquad t_1=\frac{\pi}{2\omega_-},\qquad \tau=2(t_0+t_1).$$

It follows that the coefficient of the equation (\ref{sistema_lim})  is given by  (here $\mathbbm{1}_{I}(t)$ denotes the indicator function of $I$),
$$ \beta k^2 +2 \gamma  s_{j,k}(U_\infty(t)) = A^2_+\mathbbm{1}_{I^+}(t) + A^2_- \mathbbm{1}_{I^-}(t) \qquad (0\leq t \leq \tau/2).$$

A straightforward computation (see {\em e.g.} \cite{MP} or Appendix C), shows that
$$\frac{\Delta_\infty}{2}=v_\infty(\tau)=\cos \phi_+ \,  \cos \phi_- - \frac{ a + a^{-1} }{2} \, \sin \phi_+ \, \sin \phi_- .$$

By the instability characterization   of Hill equations, if $ |\Delta_\infty| >2$, then the pure flexural solutions $u(t;q)$ are linearly unstable
(thus unstable)
for  sufficiently large $q$.
\end{proof}

%}
{\section{The MMK  model: instability tongues and some numerical results}}
\label{approx al contrario}

A classical problem in  the theory of Hill equations consists in  studying    the parametric resonance of equations of  the type
\begin{equation}
\label{moltiplicativo}
\ddot{v}       +     (\beta+qp(t))) v =0.
\end{equation}

Here $p(t)$ is a   periodic function of  period $\pi$ (just for fixing one), and $\beta$, $q$ are real parameters. The general picture is well known and can be briefly described as follows  \cite[ch. II, Th. 2.1]{MW}  \cite[ch. 2, Th. 2.3.1]{E}.
In the $(q,\beta)$-plane a sequence of separated regions  of instability (instability tongues, Arnold's tongues) emerge from the points $(0,n^2)$, $n=1,2,...$ on the $\beta-$axis which  are  bounded by two curves  $ \beta=  \beta_n^{\pm}(q)$. For values $(q,\beta)$ in the interior of the tongues, i.e  for $\beta^-_n(q) < \beta < \beta^+_n(q)$,
all solutions of (\ref{moltiplicativo}) are unstable, while outside  are stable. At the boundaries, $ \beta=  \beta_n^{\pm}(q)$,  we have
at least a non-trivial $\pi$-periodic or  $\pi$-anti-periodic solution. For some $p(t)$, and for some exceptional values of $q$,  the curves  $\beta^{\pm}(q)$  may intersect
(co-existence of periodic solutions) and form the so-called resonance pockets, as in the case of the square wave case $p(t) = \text{sign } \cos(2t)$.

The classical, and most studied, example  is the  Mathieu equation, $p(t) = \cos(2t)$, whose instability diagram was first drawn by B. VAN DER POL \& M. J. O. STRUTT in \cite{VS} and can be found on many books   \cite{McL,Arn}.
In this case we know the   order of the instability tongues as $q\rightarrow 0$, the asymptotic behavior of their width as $n \rightarrow \infty$, and many other features,
such as the fact that the width of the stability bands shrinks to zero as  $q \rightarrow \infty$, so that    the areas of instability invade the whole plane, see  \cite{Lo,WK}.

Another  case extensively studied is when $p(t)$ is a two-step function in the interval $[0,\pi]$  (Meissner equation). The papers  \cite{H1,SM}  provide
  very detailed results on the existence of resonance pockets   and on the asymptotical behavior of the stability boundaries.
%It is also worth mentioning  the paper \cite{BLS} regarding the large scale radial density of the stable zones.

  In \cite{MP3} we studied  a   generalization of the Mathieu equation,
\begin{equation}
\ddot{v}      +     (\beta + g(u(t;q))) v  =0,  \label{H2}
\end{equation}
 in which   the periodic coefficient $g(u(t;q))$ depends on the solution $u=u(t;q)$ of an  initial-value problem for a conservative second order  differential equation,
 \begin{equation}
\ddot{u}   +   4 u +  f(u )=0 ,  \qquad u(0)=q, \quad \dot{u}(0) =0,    \label{duff1}
\end{equation}
where  $f(r) =O(r^2)$, $r\rightarrow 0$.
Note that the Mathieu equation corresponds to the case  $f=0$ in (\ref{duff1}), and $g(u)=u$ in (\ref{H2}).
The equation (\ref{H2}) is again a Hill equation with two parameters that shares some aspects with the  Mathieu equation. Indeed the main result in \cite{MP3} concerns the order   of tangency at $q=0$ of the instability tongues:
\begin{equation}
\label{q0}
   \beta_n^+(q) -  \beta_n^-(q) =O(q^n), \quad q \rightarrow 0,
\end{equation}
which  is the same as that of the Mathieu equation, at least if $f$ and $g$ are real analytic functions near the origin. However,  two features make its  analysis different (and more difficult): its period   depends on the parameter $q$, as it is half the period of $u(t;q)$ if $g$ is an even function, and   is the same otherwise; the dependence on $q$ in (\ref{H2}) is no more linear. 
%As a result, the instability diagram for large  $q$ can be very different from the case (\ref{moltiplicativo}), as shown in Theorem \ref{main}.

The   system (\ref{eqq})-(\ref{zjk})  has  the same structure of (\ref{duff1})-(\ref{H2}), then for regular $f$, $g$ the asymptotic formula (\ref{q0}) applies.  Having at our disposal a  result for high energies (Theorem \ref{main}) and another one for low energies (formula (\ref{q0})), it would be interesting to have  a general picture at least in the significant case of the MMK restoring force (\ref{slack}), which   is a simple non-trivial model satisfying    the Assumptions  ($\bf S_0$)-($\bf S_1$)-($\bf S_2$), being not smooth in a single  point $r_0$.

We now come to the presentation of the   models adopted in our  numerical simulations. We focus only on the system (\ref{eqq})-(\ref{zjk}) in the case when function  $f$  is the MMK function (\ref{slack}), and on
the system (\ref{eqqbase})-(\ref{zjjbase}) below, which is a non smooth variation of it  in the case $j=k$. We performed our simulations both  with a simple set of parameters with  the same order of magnitude (Figure 1), and
with the   set   as in \cite{F,F2} where   the mechanical parameters of the Tacoma narrow bridge (TNB) and of other significant bridge models are listed,  including the MMK model, and  several numerical experiments are performed.

The analytic expression of the functions $  f_j$, owing to  (\ref{ftildaj}) in Lemma  \ref{lemmaftilde}, is known once we provide the projected form $\tilde{f}$ of $f$ in (\ref{slack}). This was already done in our previous paper \cite{MP}, and its closed form is given by,
\begin{equation}
\label{Mooretilde}
\tilde{f}(r)=mr, \quad r\geq -r_0, \quad
\tilde{f}(r)=-\frac{2 }{\pi}mr_0 \left[ \frac{r}{r_0}{\rm asin} \left( \frac{r_0}{r} \right) +\sqrt{1-\left( \frac{r_0}{r} \right) ^2}\right], \quad r\leq -r_0.
\end{equation}

As a consequence, for $j=k$, also the the functions  $  g_{j,j}=  f'_j$ (see  Lemma \ref{lemmaftilde}) have a simple closed form.
Much more complicated is the derivation of the functions $  g_{j,k}$ in (\ref{gjk}) when $j\neq k$. Their complete computation   is provided  in
Appendix \ref{formula gjk_Moore}.

In order to compare different functions, modeling the  same slackening regime, we introduce only for the case $j=k$ a non smooth variation of $f_j$, and $f'_j$ which
maintains the same shape as the original MMK model but which is adjusted in a way to have  the same asymptotic behavior of $f_j$, that is  $\lim_{r\rightarrow -\infty}\tilde f(r) = -\frac{4}{\pi}mr_0$. More precisely, we replace the function $\tilde f$ in (\ref{Mooretilde}) with  the function
\begin{equation}
\label{Moore1}
\bar f(r)=m\left[(r + 4r_0/\pi)^+ - 4r_0/\pi  \right]
\end{equation}

Then $ f_j$ can be approximated, according to Lemma \ref{lemmaftilde},
 with
%\begin{equation}
%\label{fbarp}
$$
 \bar{f}_j(r) =\frac{1}{2}\,\bar{f}_o(r), \qquad (\text{even } j), \qquad \bar{f}_j(r) =\frac{1}{2}(\bar{f}_o(r) +\frac{1}{j}\bar{f}_e(r) ), \qquad (\text{odd } j),
 $$
% \end{equation}
 and we introduce the system,
 \begin{equation}
\label{eqqbase}
\ddot{u}   +\alpha j^4 u  + 2\bar f_j(u )= 0,\,\,u(0)=q,\,\,\dot{u}(0)=0
\end{equation}
\begin{equation}
\label{zjjbase}
\ddot{v}   + (\beta j^2 +  2 \gamma \bar  f'_j (u ) )v    = 0,
\end{equation}

Both $u $ and $v $ can be  calculated explicitly and  in Appendix \ref{conti Moore base} we provide the   formula for the instability discriminant  $\Delta$ for every $(q,\beta)$.  The   expression of $\Delta (q,\beta)$ is very complicated and is actually of little help from the analytical point of view, but   it can be used to represent very quickly, using Matlab, the instability diagram in order to make a comparison with the numerical results regarding the system (\ref{eqq})-(\ref{zjk}).

Before proceeding with the discussion of the numerical results, let us fix the  starting points $\, \beta_N^+(0) = \beta_N^-(0) $  of the instability tongues for our equations.   Recall that we use $\beta$ as spectral parameter.
\begin{proposition}
\label{tipping}
Let us suppose that the function  $f$ satisfies the assumption ($\bf S_0$).
Then instability tongues  for the Hill equation (\ref{zjk}) stem from the values of $\beta_N(0)$, $N=1,2...$, in the $\beta$-axis, given by
\begin{equation*}
\label{betaponte}
\beta_N(0)=\frac{ \alpha j^4+2m}{4k^2}\,N^2-\frac{2  \gamma m}{k^2}.
 \end{equation*}

If $j$ is even the tongues corresponding to an odd index $N$ disappear, since the actual period is half the period of $u(t;q)$.
\end{proposition}

We skip the proof which, after  a translation of the spectral parameter $\beta$, follows from  the fact that, if
$\tau(q)$ is the period of $u(t;q)$, then its limit as $q \rightarrow 0$, is given by
$ \tau_0= 2\pi/\sqrt{\alpha j^4 +2m} $.
The different behavior for  even or odd  $j$    arises from the fact that for even $j$, $ g_{j,k}(u(t))$ is an even function and the minimal period of the equation  (\ref{zjk}) is half the period $\tau(q)$ of $u(t;q)$.

\begin{figure}[h]
   \begin{minipage}[b]{0.5\linewidth}
    \centering
    \includegraphics[width=.8\linewidth]{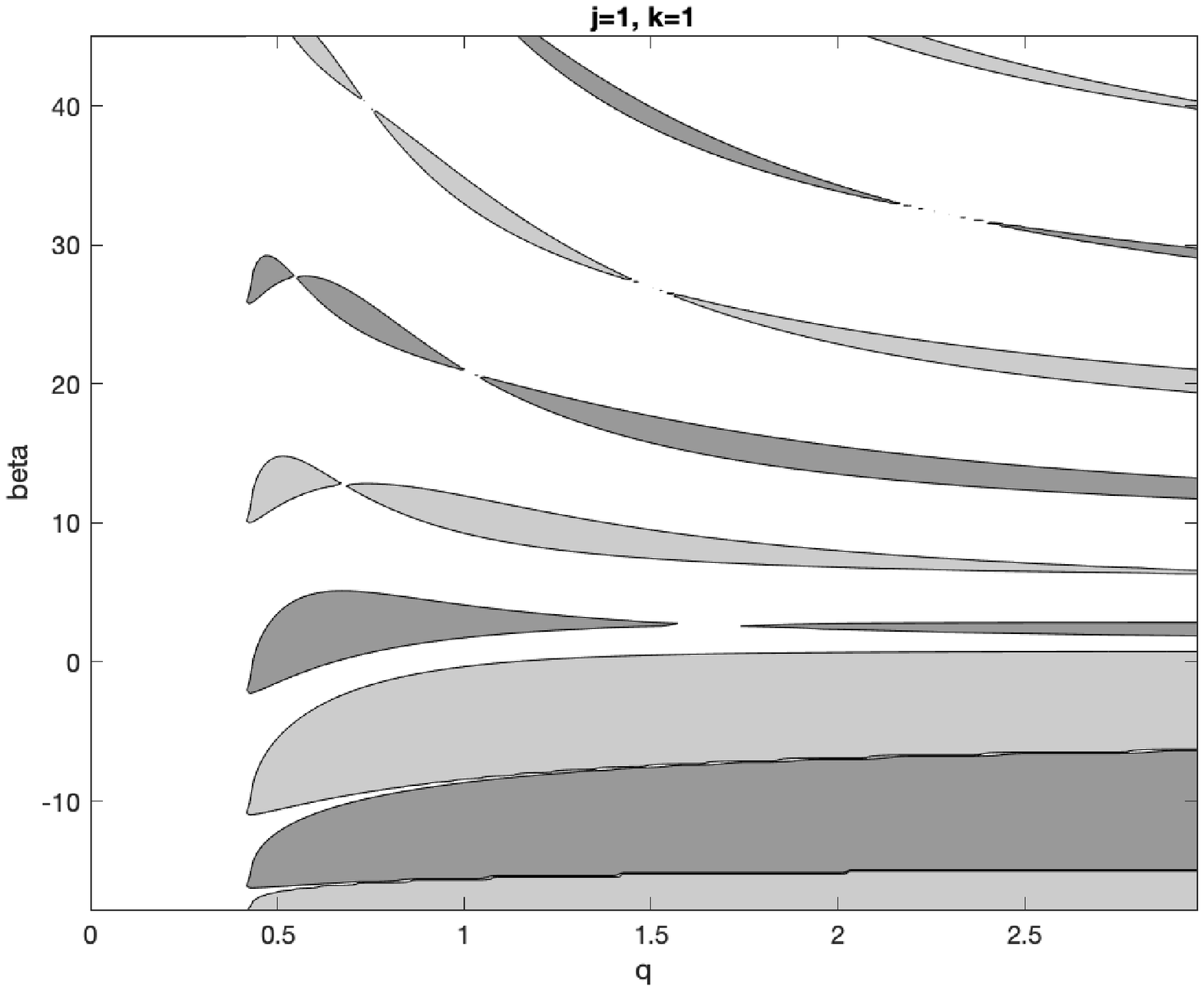}
  \end{minipage}%%
  \begin{minipage}[b]{.5\linewidth}
    \centering
    \includegraphics[width=.8\linewidth]{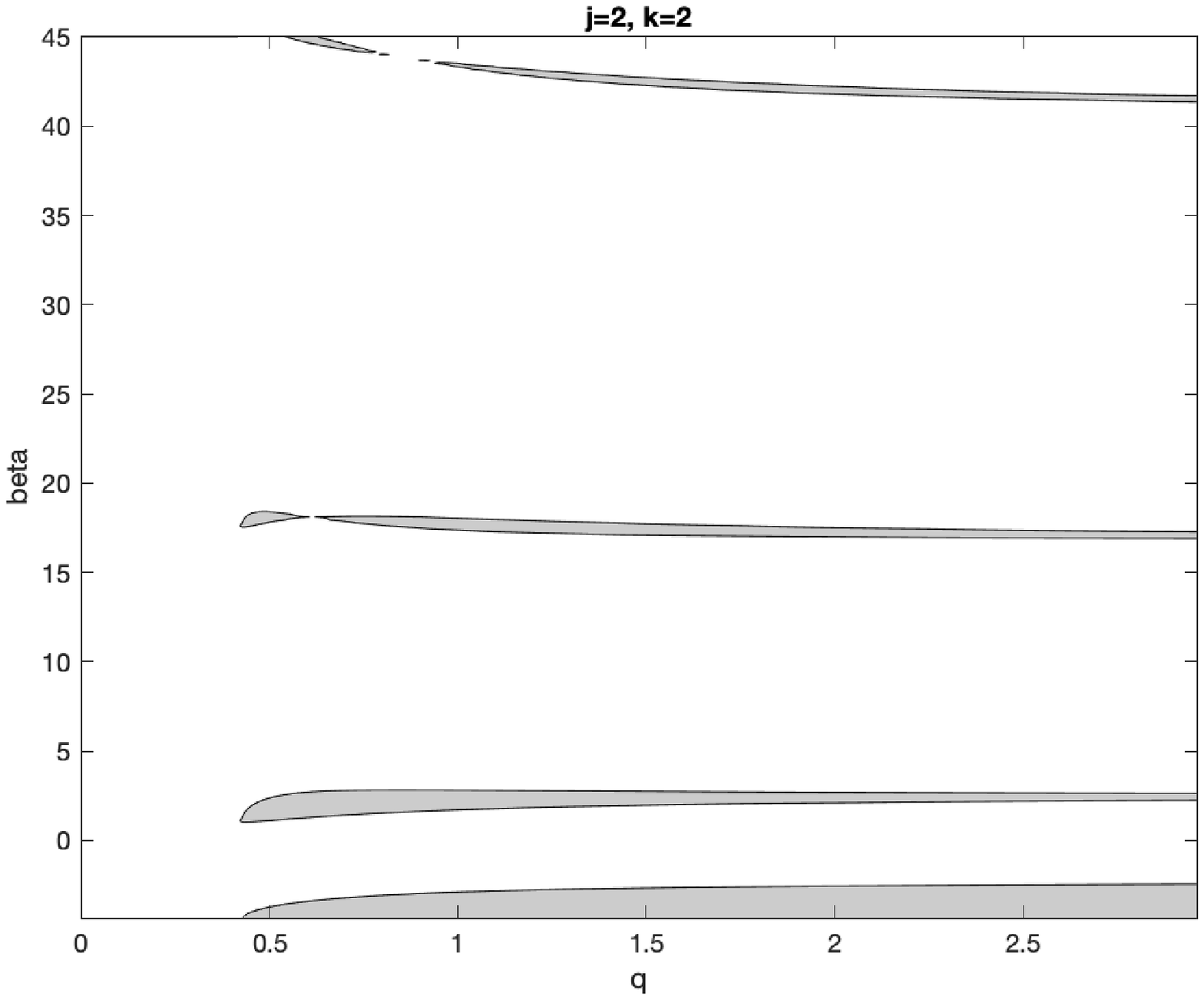}
  \end{minipage}
  \begin{minipage}[b]{0.5\linewidth}
    \centering
    \includegraphics[width=.8\linewidth]{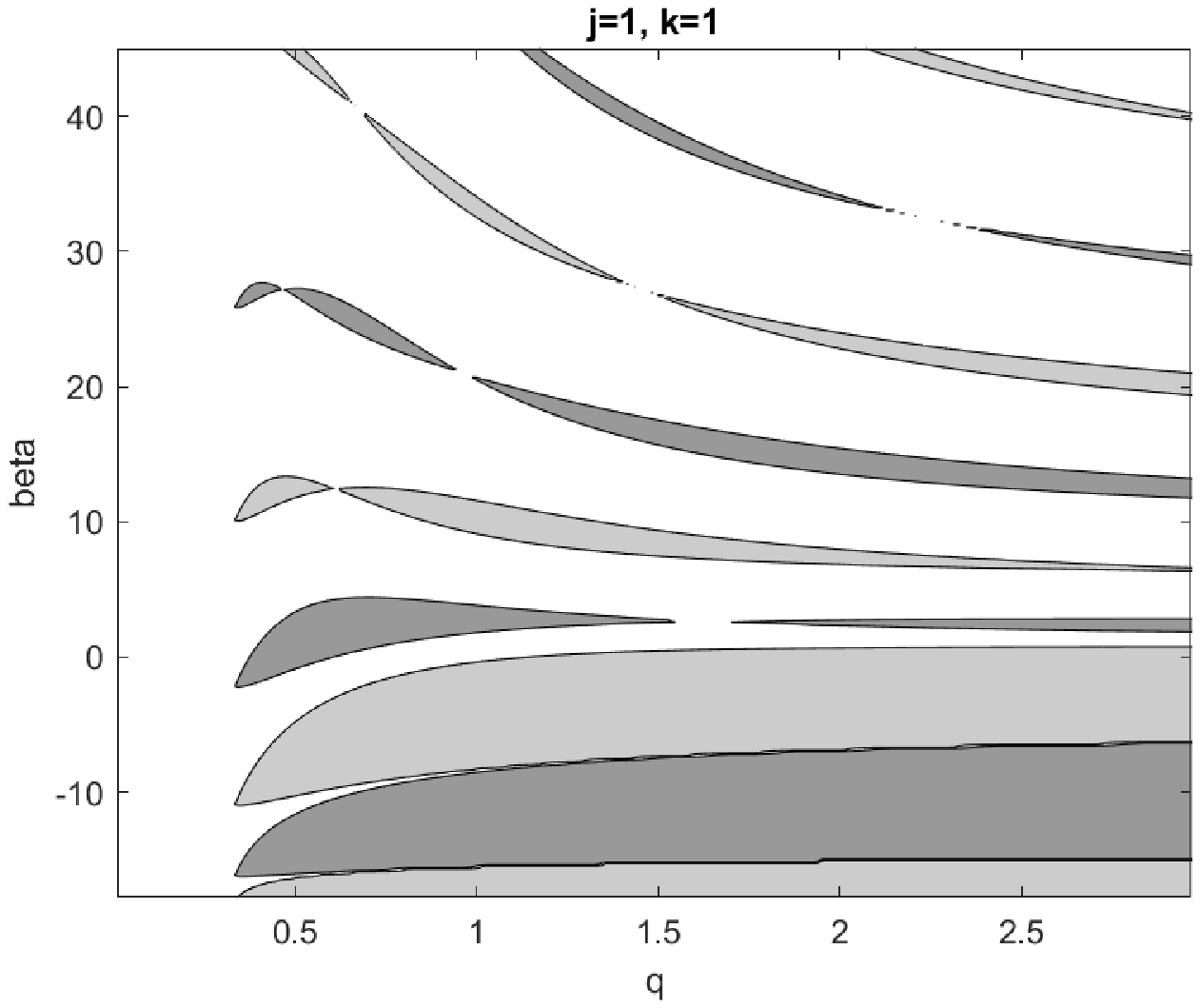}
  \end{minipage}%%
  \begin{minipage}[b]{0.5\linewidth}
    \centering
    \includegraphics[width=.8\linewidth]{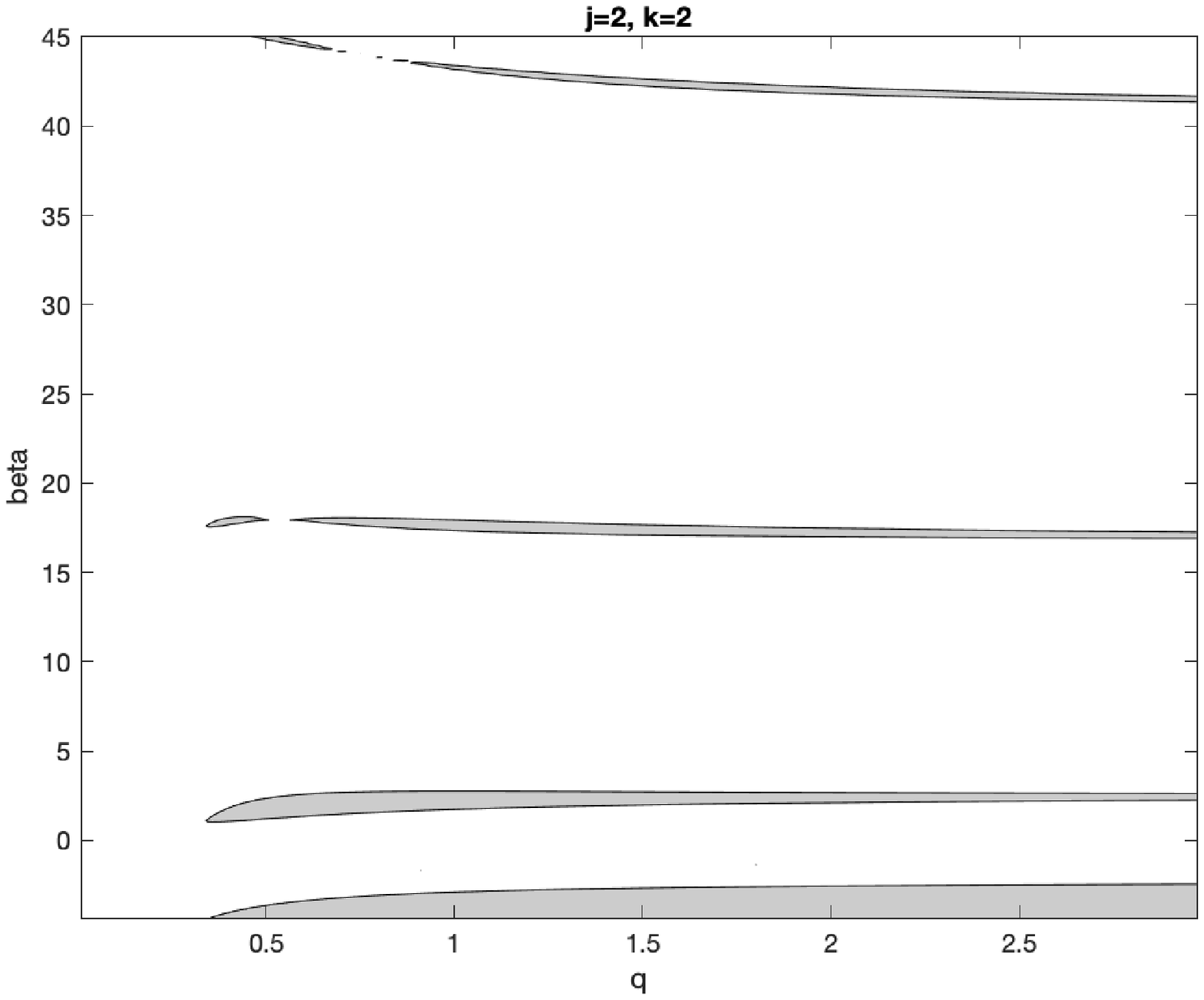}
  \end{minipage}
  \caption{Instability diagrams of systems (\ref{eqqbase})-(\ref{zjjbase}), first line, and (\ref{eqq}, \ref{zjk}), second line.
  The fixed parameters are: $\alpha =1$, $\gamma =3$, $m=3$, $r_0=1/3$.
  $\Delta (q,\beta) >2$ in t
  he light grey zones, and   $\Delta (q,\beta) <-2$ in the dark grey zones.
   Note that on the right column the tongues corresponding to odd index vanish.}
    \label{figu1}
\end{figure}

\begin{figure}[h]
   \begin{minipage}[b]{0.5\linewidth}
    \centering
  \includegraphics[width=.8\linewidth]{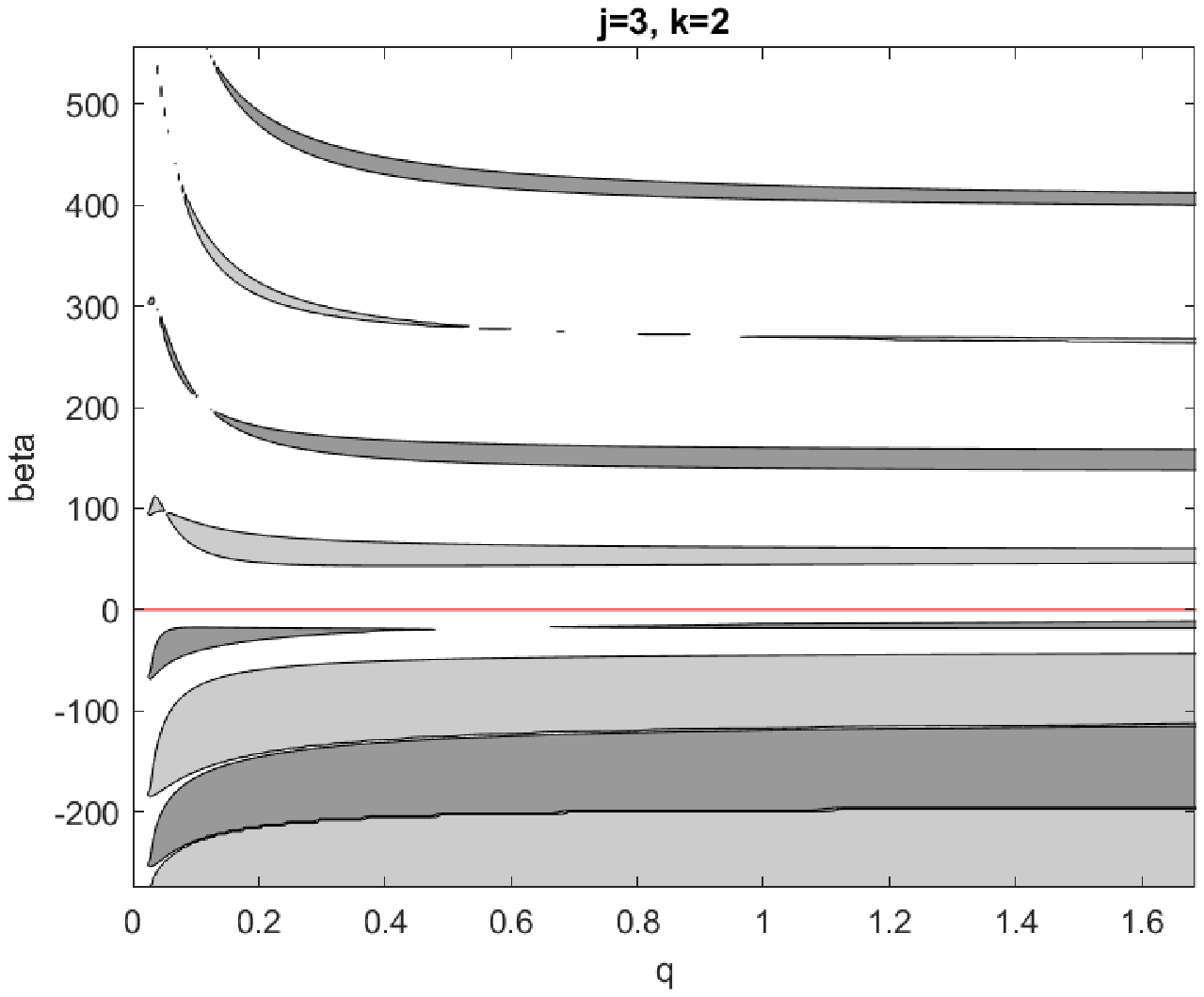}
  \end{minipage}%%
  \begin{minipage}[b]{.5\linewidth}
    \centering
  \includegraphics[width=.8\linewidth]{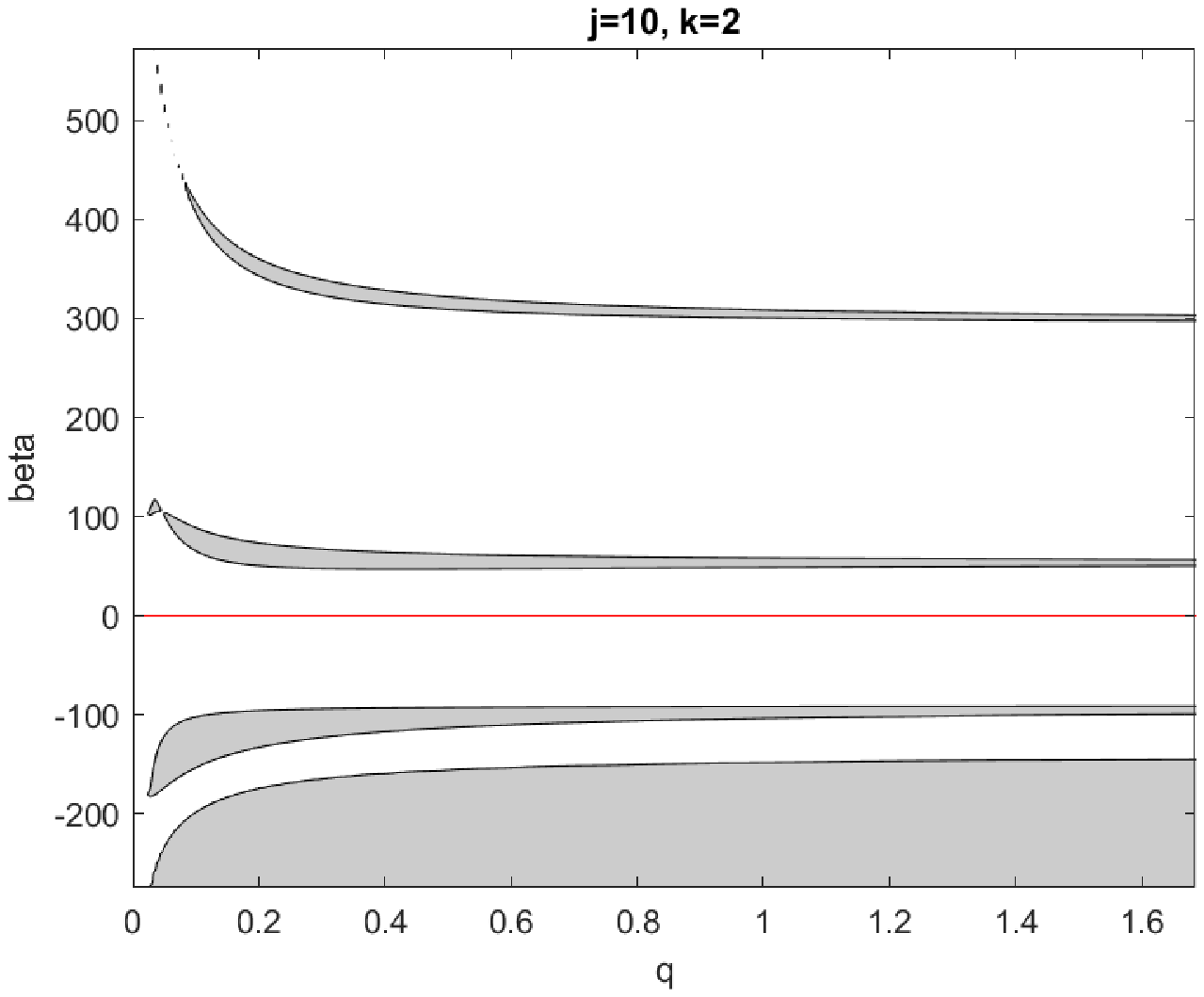}
  \end{minipage}
  \caption{Instability diagrams of system (\ref{eqq})-(\ref{zjk}) for  the MMK function, and $j \neq k$. The fixed parameters (TNB) are: $\alpha =8.0353\cdot 10^{-4}$,  $m=185.1$, $r_0=0.0265$.
 %\caption{Graph of the instability tongues for$j \neq k$. The fixed parameters are: $\alpha =1$, $\gamma =3$, $m=3$, $h=1$.
%The absolute value of the instability discriminant is greater than $2$ in the yellow zones, and lesser than $-2$ in the white zones.
 The actual value $\beta=8.1833\cdot10^{-5}$ of the TNB   is highlighted.
}
 \label{fig2}
\end{figure}

In the first line of Figure \ref{figu1} we show the diagram of the instability tongues for the "irregular approximation" (\ref{eqqbase})-(\ref{zjjbase}) when $j=1$ (left), $j=2$ (right). In the second line of Figure \ref{figu1}  the diagrams of the corresponding instability tongues for the system  (\ref{eqq})-(\ref{zjk}), ($\tilde{f}$ as in (\ref{Mooretilde})), with $j=k$ and with the same parameters,  in order  to compare the diagrams with the first line.
It is worth noting that both systems are linear decoupled systems for low values of $q$, then we have  a trivial case of coexistence of periodic orbits, since
all solutions are periodic.

Coming to the specific features  of the system (\ref{eqqbase})-(\ref{zjjbase}),   the true instability tongues split from the points $(\frac{4}{\pi}r_0, \beta_N(0))$, where $\beta_N(0)$ is as in Proposition (\ref{betaponte}).
We note   the persistence of areas of stability even for large values of $q$, so that their  widths   do  not   shrink  to $0$, as $q\rightarrow \infty$. This  is more evident for greater values of $j$ and  is in accordance with Theorem \ref{main}.
In this regard, it is worth comparing   the instability diagram in Figure  \ref{figu1}, and the one when $p(t)$ is a step function in  problem  (\ref{moltiplicativo}), see \cite{SM}. In both cases $\beta$ has  the same meaning, but in the first case   the parameter $q$ acts both on the period and the width of the steps of the Meissner equation, while in the second case, the parameters $q$ in (\ref{moltiplicativo}) acts only on the height of the steps, as the  period and width of the steps are fixed. Then
 we note that   the width  at  the splitting points $(\frac{4}{\pi}r_0, \beta_N(0))$, unlike (\ref{q0}),  is $O(q)$, $q\rightarrow 0$.  This is due to the singularity of $\bar f_j$ in
 $ r= 4\pi/ r_0$.
Finally we  observe the onset of the so called resonance pockets, that are typical of some  Hill equations of Meissner type, see  \cite{H1,SM}.

The diagrams in the second line of Figure \ref{figu1} are obtained by the numerical solutions of  the system (\ref{eqq})-(\ref{zjk}). We     used the Matlab solver {\tt ode23t} with  a quite high  {\tt reltol} and  {\tt absTol} setting.  Even though we used the closed formulas for $f_j$, $g_{j,k}$ provided in Appendix B, a very long computation time was needed
to obtain an accurate value of $\Delta (q,\beta)$, and a reasonably good quality diagram.     Indeed, at every step in the grid value of $q$, a precise evaluation of the period of $u(t;q)$ was necessary.

Figure \ref{fig2} shows two instability diagrams of the system (\ref{eqq})-(\ref{zjk}) respectively for
$j=3$, $k=2$ on the left, and $j=10$, $k=2$ on the right. In this case  $j\neq k$,  so that  the approximated    system (\ref{eqqbase})-(\ref{zjjbase}) is no more available.

More precisely, we decided to use the set of structural parameters of the Tacoma Narrow bridge (TNB) reported in \cite{F} that are derived from the Technical Report (1941) of AMMANN et al. \cite{AKW}.
The mechanical parameters of the TNB
are characterized by very different orders of magnitude. Nonetheless the general picture of the instability tongues (behavior at the splitting points and resonance pockets) remains the same and seems to be a constant feature of the MMK model.

The reports about the failure of the TNB show that the transfer from flexural to torsional energy regarded most probably the $9-2$ or the $10-2$ modes (see for example \cite{Gaz} for an extensive owerwiew). We checked all the modes $j$-2, for $2\leq j\leq 10$, and $0<q\leq 1.6$ m (flexural oscillations of amplitude of about one and a half meters were reported by witness, see for example again \cite{F}, Appendix A).
For our simplified model the tongues corresponding to an even $N$ are relatively thick in the given interval for $q$, but stay far away from the significant value of $\beta$. Whereas the tongues related to an odd $N$ get very thin as $j$ grows, and also such modes do not present any instability near the significant value of $\beta$. We point out that the fish-bone model does not take into account the dynamics of the suspension cables and their mechanical parameters, then our simulations are interesting only from a mathematical point of view.

\appendix
%%%%%%%%%%%%%%%%%
%%% APPENDIX  varie
%%%%%%%%%%%%%%%%%

\section{Regularity  of $\psi_1$ and $\psi_2$ }
\label{C1}

% LEMMA NON DEL TUTTO SODDISFACENTE
\begin{lemma}
\label{derivata}
Let $f$ be a real valued function such that  $\, f\in C^0(\R) $, $f\in C^1(\R \setminus \{r_0\})$, and  $ \quad \sup_{ \R \setminus \{r_0\}} |f'(x)| < \infty$, $r_0 \in \R$. Let $I$ be the closed interval $[a,b]$, $\, u \in C^1(\R^{n+1})$,  $w \in L^\infty(I)$, and define
$$ G(p) = \int_{I} f(u(x,p))w(x)\,  dx, \qquad  (p\in \R^n). $$

If, for each $\, p \in \R^n$,
\begin{equation}
\label{b}
 u(x,p) =r_0, \text{ has (at most) a finite number of solutions} \quad  x \in I,
\end{equation}
then $\,G \in C^1(\R^n)$,  and the following differentiation formula holds for every $p\in \R^n$, $$ \bigtriangledown G(p) = \int_{I} f'(u(x,p))\,  \bigtriangledown_p u(x,p) w(x)\,dx.  $$

\end{lemma}

\begin{proof}
For simplicity we assume $n = 1$. The general case $n \geq  1$ follows exactly in the same way.

%in the case when $w \not \equiv 1$ there are substantially obvious changes which we %briefly outline at the end.
We set
$$ \Gamma = \{(x,p) \in \R^2: \, u(x,p)=r_0\} , \qquad \Omega = \R^2\setminus \Gamma,$$
and let  $p_0\in \R$ be a number  such that there exists a  finite number of   points  $(x_j,p_0)\in \Gamma$, $j=0, 1, ..., m$. Possibly by decomposing  the interval $I$  into a finite union of intervals $I_j$  in which $u(x,p)=r_0$ has a single solution, we may assume that there exists a unique point $(x_0,p_0) \in \Gamma$.

Let us fix $\vare >0 $ and define
$$ K_\vare^0 = \{ (x,p_0):\, x\in I,   |x-x_0| \geq \vare\}. $$

Since $K_\vare^0$ is compact, its distance from the closed set $\Gamma$ is strictly positive. Then there exists $\delta = \delta(\vare)$, such that
$$ K_\vare^\delta  = \{ (x,p):\, x\in I,  |x-x_0| \geq \vare, |p-p_0| \leq \delta \}  \subseteq \Omega. $$

By the continuity of $u$, we have either  $u <r_0$ or $u> r_0$ on each connected component of $ K_\vare^\delta $, thus $f\circ u \in C^1(K_\vare^\delta)$.

Let us set $$\, g(p) =  \int_{I} f'(u(x,p))\,  \frac{\partial u}{\partial p} (x,p)w(x) \,dx, $$
 and  split the function $G$ in the following way:
$$ G(p) = G_{1,\vare}(p) + G_{2,\vare}(p) \qquad (|p-p_0| \leq \delta),$$
where
$$ G_{1,\vare}(p) =  \int_{x\in I,  |x-x_0|\geq \vare} f(u(x,p))w(x)\,  dx, \qquad G_{2,\vare}(p) = \int_{x\in I,  |x-x_0|< \vare} f(u(x,p))w(x)\,  dx , $$
so that $G_{1,\vare} \in C^1([p_0-\delta,p_0+\delta])$ by the standard rule for differentiation under the integral sign.   As for the second term we have,
\begin{eqnarray}
\label{lip}
 \left | \frac{G_{2,\vare}(p_0 +h) -G_{2,\vare}(p_0)}{h} -  \int_{x\in I,  |x-x_0|< \vare} f'(u(x,p_0)) \frac{\partial u}{\partial p} (x,p_0) w(x) dx \right|   & \leq  & \nonumber  \\[10pt]
 &  &  \hspace{- 12 cm}  \int_{x\in I,  |x-x_0|< \vare} \left|  \frac{f(u(x,p_0 +h)) -f(u(x,p_0 ))}{h} w(x) \right|  + \left| f'(u(x,p_0))  \frac{\partial u}{\partial p} (x,p_0) w(x) \right|\, dx  \nonumber
% \qquad   ( =: A_h)
\end{eqnarray}

Since $f$ has bounded derivative  (thus Lipschitz continuous), $w\in L^\infty$, and $u\in C^1$, the sum of both terms in the last integral  is bounded above by some positive constant $M$. In conclusion, by
considering the incremental ratio of both $G_{1,\vare}$, and $G_{2,\vare}$ we get
$$ \limsup_{h\rightarrow 0}   \left| \frac{G(p_0 +h) -G(p_0)}{h}-   g(p_0) \right| \leq 2M\vare, $$
for every $\vare >0$,  which proves that $G'(p_0) = g(p_0)$.

Now we prove the continuity of $g(p)$. With the same notations as before, for any fixed $\vare>0$, we have
$$ g(p) = G'_{1,\vare} (p) +  \int_{x\in I,  |x-x_0|< \vare} f'(u(x,p)) \frac{\partial u}{\partial p} (x,p) w(x) \, dx
\qquad (|p-p_0|\leq \delta), $$
where $G'_{1,\vare} \in C^0([p_0-\delta,p_0+\delta])$. As for the other term, as before the boundedness of $f'$, and the
regularity of $u$, yield
$$ \int_{x\in I,  |x-x_0|< \vare} \left| f'(u(x,p)) \frac{\partial u}{\partial p} (x,p) w(x) \right|  dx  \leq 2M\vare \qquad  (|p-p_0|\leq \delta).$$

Therefore we infer that
$$
  | g(p) - g(p_0)  |    \leq   | G'_{1,\vare}(p) -  G'_{1,\vare} (p_0) | + 4M\vare ,
$$
so that
$$ \limsup_{p\rightarrow p_0}  | g(p) - g(p_0)  |  \leq  4M\vare, $$
which concludes the proof of the Lemma.

\end{proof}

%%%%%%%%%%%%%%%%

\noindent {\bf Proof of Proposition \ref{psi}}. We apply  Lemma \ref{derivata} to $f$ satisfying ($\bf{S_0}$), $I=[0,\pi]$, $p=(y,z) \in \R^2$, $u(x,p) = y\sin(jx)\pm z\sin(kx)$, $ w(x) = \sin(jx)$ or $ w(x) = \sin(kx)$. The lemma has a local character and it is enough to verify the hypotheses in a neighborhood of each point $p_0= (y_0,z_0) \in \R^2$, therefore we may  assume that the function $f$ is not differentiable at a single point
$r_0$, and that has bounded derivative  elsewhere. The regularity of $u$, and $w$  being obvious, we must verify  the assumption (\ref{b})
for $r_0 \neq 0$ (recall that this was required in ($\bf{S_0}$)).
By contradiction, if for a fixed $(y,z)\in \R^2 $, the equation
$ u(x,y,z)= r_0$ were satisfied for an infinite set of $x\in I$, since $x\mapsto u(x,y,z)$ then $u(x,y,z)= r_0$ for every $x \in \R$. This follows by the unique continuation principle applied to the analytic function  $x\mapsto u(x,y,z)$. In particular we would have $u(0,y,z) = 0 = r_0$, which is a contradiction.

\

We conclude this Appendix with a few  remarks on the differentiation Lemma \ref{derivata}.

The lemma, in spite of its simplicity,    cannot be derived from the traditional  Lebesgue theorem of differentiation under integral, as one can easily verify. A simple example is provided by the function $\, G(p) = \int_0^1 |x-p| \, dx $, which is differentiable everywhere, and of class $C^1$.  As a matter of fact,   Lemma \ref{derivata} is more a regularity result  than a sufficient condition to differentiate under integral sign.

The condition (\ref{b}) serves to our purposes but can be easily relaxed, for example by assuming that the set $\{x\in I: \, u(x,p) =r_0\}$   has (at most) a finite number of accumulation points for any fixed $p$.  A trivial example in which this last condition is violated, for $p=0$, is given by $G(p) = \int_0^1 |p| \, dx$, which of course is not differentiable at $p=0$. In this regard, we point out the relevance  in the application of the Lemma  to the functions $\psi_1$, and $\psi_2$, in which $u(x,p) = y\sin(jx)\pm z\sin(kx)$, to the assumption $r_0\neq 0$.
On the other hand, any relaxed version of condition (\ref{b}) is far from being necessary, as the simple example $G(p) = \int_0^1 |p^3| \, dx$ shows.
In   literature there are other more or less classical differentiation  results, see \emph{e.g. } \cite{S}, but we have found that for our purposes the verification of the hypotheses would have been more difficult than the direct proof of the lemma.

\section{Explicit formulas  for the MMK function}
\label{formula gjk_Moore}

The non-linear terms of the system (\ref{ODE}) and the periodic coefficient of the Hill equation    (\ref{zjk})   are both defined as an integral.
This fact becomes very time consuming  when it comes to drawing the   instability tongues with Matlab, as the solution of some ten thousands of systems is  needed.
So, in the case of the MMK slackening model (\ref{slack}), we decided to provide  the explicit computation of   $g_{j,k}(r)$ in (\ref{zjk}),  for every choice of $j$ and $k$.
The result in Lemma \ref{teorematheta} for $r_0=0$  is  necessary   for the computation of the limit system in Section \ref{high energies}.

We need a simple computational lemma.

\begin{lemma}
Let $j$, $k$ be two positive integers, then
\begin{equation}
\label{q}  q(j,k) : =\sum_{n=1}^{j}{\cos(\frac{2k}{j}n\pi)} \, = \,  \begin{cases}
     j & \text{if  } \quad k/j \in \N, \\
     0 & \text{otherwise}.
     \end{cases}
\end{equation}
\begin{equation}
\label{p} p(j,k):=\sum_{n=1}^{j}{(-1)^n\sin(\frac{2k}{j}n\pi)}=\begin{cases}
     0 & \text{ if  }  j \text{ even }, \\
      -\, \tan(\frac{k}{j}\pi) & \text{ otherwise}.
\end{cases}.\end{equation}
\end{lemma}
\begin{proof}
Let us prove (\ref{q}). In the case when $k/j \in \N$, we obviously have $\, q(j,k)= j $; otherwise
we get  ($i= \sqrt{-1}$, $Re=$ real part),
$$ q(j,k) = \sum_{n=0}^{j-1}{\cos(\frac{2k}{j}n\pi)} =Re \, \sum_{n=0}^{j-1} e^{\frac{2k \pi  i}{j}n}  = Re \, \frac{1 - e^{2k \pi i }}{1- e^{ 2k \pi i/j}}=0. $$

In order to prove (\ref{q}),  we observe that, if   $\frac{2k}{j}\in \N$, then $p(j,k)=0$. Otherwise, we have ($Im$ = imaginary part),
 $$ p(j,k) = \sum_{n=0}^{j-1}{(-1)^n\sin(\frac{k}{j}2n\pi)}= Im \, \sum_{n=0}^{j-1} (-e^{\frac{k}{j}2\pi i})^n= Im \, \frac{1-(-1)^j}{ 1+ e^{\frac{k}{j}2\pi i}},$$
 thus $\, p(j,k) =0\,$ for $\, j $ even. When $j$ is odd, we get
$$p(j,k)=-\, \frac{\sin (\frac{k}{j}2\pi)}{1+\cos(\frac{k}{j}2\pi)}= -\, \tan(\frac{k}{j}\pi).$$

\end{proof}

The closed form of the function $\, g_{j,k} \,$, in the case of MMK function  (\ref{slack}), is given by $\, 2m/\pi$ times the function $\, H_{j,k}(r)\, $ defined here below.
\begin{lemma}
\label{teorematheta}
Let $j$, $k$ be positive integers, $r_0\geq 0$,  and  let $H_{j,k}(r)$ be the function defined as ($H[\cdot] = $ Heaviside step function),
%\begin{equation}
%\label{gtildeMoore}
$$
H_{j,k}(r) := \int_0^{\pi}H( r\sin jx+r_0) \, \sin^2 kx \ dx \qquad (r\in \R).
$$
%\end{equation}

Let us set  $\; \theta(r)={\rm asin}(r_0/|r|)$, for $|r|\geq r_0$.

Then, for every $j$, $k$, we have
$$\, H_{j,k}(r)= \frac{\pi}{2} \qquad \text{ for } \; |r|\leq r_0, $$
while  for $|r|> r_0$   (note that for even $j$  the second row vanishes),
\begin{eqnarray}
 \label{Hjk}
% H_{j,k}(r)  &= & \frac{\pi}{4}+ \frac{\theta(r)}{2} - \frac{q(j,k)}{4k} \, \sin\left( \frac{2k\theta(r)}{j}\right), \qquad j \text{  even}; \\
H_{j,k}(r)  & = & \frac{\pi}{4}+ \frac{\theta(r)}{2} - \frac{q(j,k)}{4k} \, \sin\left( \frac{2k\theta(r)}{j}\right)  + \nonumber \\
 & &   \hspace{-2cm}
 \, + \, \frac{1-(-1)^j}{2}\,  \frac{r}{|r|} \left[\frac{\pi}{4j} - \frac{\theta (r)}{2j}  +\frac{p(j,k)}{4k }\cos\left(\frac{2k \theta(r)}{j}\right)+ \frac{1}{4k}\sin\left(\frac{2k\theta(r)}{j}\right) \right].
 %\qquad j \text{  odd}.
  \end{eqnarray}
  \end{lemma}

\begin{proof}
If $|r|\leq r_0$, we have     $ \, r\sin jx +r_0 \geq 0$,    thus $ H_{j,k} (r) = \int_0^{\pi}\sin^2 kx \ dx=\pi/2$.

The other cases are much more involved. First of all, by changing the integration variable, we write
$$H_{j,k} (r) =\frac{1}{j}\int_0^{j\pi} H(r \sin z+r_0)  \sin^2 (kz/j) \,dz; $$
and  observe that, for $\, z\in [0,2\pi]$, we have
$$\, r \sin z+r_0 \geq 0 \; \text{ iff }   \, z\in B^+_0(r) := [0 ,\pi + \theta (r)]\cup  [2\pi - \theta (r),2\pi]  \qquad (r>r_0), $$
$$\, r \sin z+r_0 \geq 0 \; \text{ iff }   \, z\in B^-_0(r) := [0,  \theta (r)]\cup  [\pi - \theta (r),2\pi]  \qquad (r<- r_0);  $$
then we set $\, B^{\pm}_n(r) = B^{\pm}_0(r) + 2n\pi $ (the translated sets).

We begin with $r>r_0$, and $j$ even.
 We get
$$H_{j,k}(r) = \frac{1}{j} \sum_{n=0}^{j/2-1} \int_{B^+_n(r)} \sin^2 (kz/j) \,dz; $$  by direct computation of the integrals, we obtain
\begin{eqnarray*}
\frac{1}{j} \sum_{n=0}^{j/2-1} \int_{B^+_n(r)} \sin^2 (kz/j) \,dz & = & \frac{\pi}{4}+\frac{\theta}{2}+   \frac{1}{4k}  \sum_{n=0}^{j/2-1} \left( \sin(4kn\pi/j) -  \sin(4k(n+1)\pi/j)\right)  \\
 &  &  \hspace{-1cm}   + \frac{1}{4k} \sum_{n=0}^{j/2-1} (\sin(\frac{2k}{j}((2n+2)\pi-\theta)) - \sin(\frac{2k}{j}((2n+1)\pi+\theta)).
 \end{eqnarray*}

The first summation on the right hand side is telescopic and  cancels out. The last summation, by using the trigonometric addition formula, may  be written as \\
%$$=\frac{1}{j}\displaystyle{\sum_{n=0}^{j/2-1}(\pi/2+\theta + \frac{j}{4k}(\sin(\frac{2k}{j}((2n+2)\pi-\theta)) - \sin(\frac{2k}{j}((2n+1)\pi+\theta))}=$$
%NOTA: $m$\emph{ non va bene come indice di sommatoria}
$$  \frac{1}{4k}\left(\sum_{n=1}^{j}{(-1)^n\sin(\frac{2k}{j}n\pi)}\right)\cos(\frac{2k}{j}\theta)-
\frac{1}{4k}\left(\sum_{n=1}^{j}{\cos(\frac{2k}{j}n\pi)}\right)\sin(\frac{2k}{j}\theta).$$

Owing  to (\ref{q}), and (\ref{p}), and  putting together the various contributions,  this concludes the proof in the case  $j$  even and  $r>r_0$.

If $j$ is even and  $r<-r_0$,  we replace $B^+_n(r)$ with $B^-_n(r)$ and  follow  the same procedure, to obtain
$$H_{j,k} (r) =\displaystyle{\frac{\pi}{4}+\frac{\theta}{2}-\frac{p(j,k)}{4k}\cos(\frac{2k}{j}\theta)-\frac{q(j,k)}{4k}\sin(\frac{2k}{j}\theta)} =\displaystyle{\frac{\pi}{4}+\frac{\theta}{2}
-\frac{q(j,k)}{4k}\sin(\frac{2k}{j}\theta)} .$$

Let us now consider the case when
 $j$ is odd.  In the simple case  $j=1$, the direct computation  of $H_{j,k}(r)$ is trivial, and yields
 $$ H_{1,k}(r) = \frac{\pi}{2}, \quad \text{ if } \, r>-r_0; \qquad H_{1,k}(r) = \theta (r) -  \frac{\sin(2k\theta (r)) }{2k}
 \quad \text{ if } \, r<-r_0,$$
 which coincides with the formula (\ref{Hjk}) (although its verification is somehow hidden).

 For $j$ odd, and $j\geq 3$, and $r>r_0$, we get, again by following the same procedure as above,
% the subintervals where $H((r \sin jx)+r_0)=1$ are $(j+1)/2$ and the others are $(j-1)/2$.
\begin{eqnarray*}
H_{j,k}(r) & = & \frac{1}{j} \sum_{n=0}^{(j-3)/2} \int_{B^+_n(r)} \sin^2(\frac{k}{j}z) \, dz + \frac{1}{j}  \int_{(j-1)\pi}^{j\pi}   \sin^2(\frac{k}{j}z) \, dz   \\
 &  & \hspace{-3cm} = \left(1+\frac{1}{j}\right)\frac{\pi}{4} +\left(1-\frac{1}{j}\right) \frac{\theta}{2} + \frac{p(j,k)}{4k}\cos(\frac{2k}{j}\theta)-\frac{1}{4k}(q(j,k)-1)\sin(\frac{2k}{j}\theta);
 \end{eqnarray*}
while,  for $r < -r_0$,
\begin{eqnarray*}
H_{j,k}(r) & = & \frac{1}{j} \sum_{n=0}^{(j-1)/2} \int_{B^-_n(r)} \sin^2(\frac{k}{j}z) \, dz - \frac{1}{j}  \int_{j\pi}^{(j+1)\pi}   \sin^2(\frac{k}{j}z) \, dz   \\
 &  & \hspace{-3cm} = \left(1- \frac{1}{j}\right)\frac{\pi}{4} +\left(1+ \frac{1}{j}\right) \frac{\theta}{2} - \frac{p(j,k)}{4k}\cos(\frac{2k}{j}\theta)-\frac{1}{4k}(q(j,k)+1)\sin(\frac{2k}{j}\theta).
 \end{eqnarray*}

Again owing to  (\ref{q}), and (\ref{p}), this concludes the proof of the lemma.
\end{proof}

If $f(r)$ is the MMK  function in  (\ref{slack}), where $r_0>0$, then $g_{j,k}(r)=\frac{2m}{\pi} H_{j,k}(r)$. We can verify by direct inspection in (\ref{Hjk}) that this is a continuous function, as expected, because Proposition \ref{psi} says that  $\frac{\partial \psi_2}{\partial z} (y,0)=2 g_{j,k}(y)$ must be continuous.
If we set instead $r_0=0$,  (\ref{Hjk})  gives us the explicit formula for  $s_{j,k}(r)=\frac{2M}{\pi} H_{j,k}(r)$, that is needed for defining the limit system in Theorem \ref{main}.
This function is not continuous in $r=0$ if $j$ is odd.

This makes clear again that, if we weaken the condition b) in the assumption  ($\bf S_0$), substituting the MMK function with the similar one $f(r)=Mr^+$, the respective functions
$ \psi_1$, $ \psi_2$ in (\ref{ODE}) are no more smooth.

 \section{The discriminant for the  approximation (\ref{Moore1}) }
\label{conti Moore base}

In this section we provide the computation of   the instability discriminant $\Delta$ for each fixed value of $\beta$ and $q$ for the Hill equation (\ref{zjjbase}).
We  find the  explicit  solution $u(t;q)$ of the equation (\ref{eqqbase}), and its period $\tau = \tau (q) $.  The periodic coefficient in  (\ref{zjjbase})
is affected only by the slope of the piecewise linear function $ \bar f_j(r)$ whose values change at some transition points of $u(t;q)$.
It turns out that, at any fixed value of $q$,   (\ref{zjjbase})  is a multi-step Hill equation or Meissner equation \cite{BES,MW}.

We recall how to compute the discriminant of a Hill equation with a positive, multi-step potential.  Let the interval $[0,\tau] = \cup_{i=0}^n I_i$ be the union of disjoint
intervals, each having length $\,\Delta t_i$. Let us consider a   potential $Q(t)$ which is    $\tau$-periodic  positive and constant   on each subinterval $I_i$, that is,
$$ Q(t) = \sum_{i=0}^n A_i^2 \, \mathbbm{1}_{I_i}(t)  \qquad 0\leq t \leq \tau.$$

The monodromy matrix, and the discriminant  of the Hill equation  $ \, \ddot{v}(t) + Q(t) v(t)=0 $,  are computed  as follows, see  \cite{BES} p. 12. If
the $L_i$'s are the transition matrices
$$
L_i=\begin{bmatrix}\cos(A_i \Delta t_i) & \frac{1}{A_i} \sin(A_i \Delta t_i)\\ -A_i \sin(A_i \Delta t_i) & \cos(A_i\Delta t_i) \end{bmatrix},\quad i=0,1, ... n,
 $$
then we get
%\begin{equation}
%\label{M}
$$
 M = L_n\,L_{n-1}  ... L_1 \, L_0, \qquad \Delta ={\rm tr} M.
 $$
% \end{equation}

To simplify some calculations, it is important to note that the discriminant is invariant with respect to any cyclic permutation of the transition matrices, so for instance
$M=L_2\,L_1\,L_0$, and $M'= L_1\,L_0 \, L_2$ are in general different matrices having the same trace. In fact, they correspond to  different translations in time of  $Q(t)$, which of course leave  the discriminant unchanged.

In applying the previous formulas to the   equation  (\ref{zjjbase}), the $A^2_i$'s coefficients are known, being determined by the constant slopes of the function    $\bar f_j(r)$. It remains to compute the length of the  intervals $I_i$, and their ordering, modulo cyclic permutations.

First of all, we observe that, in both cases $j$ even or odd, we have
$$\,  \bar f_j(r) =mr,  \quad \text{ if }   \, | r | \leq \bar r:=\frac{4}{\pi}r_0, $$
therefore when  the initial value $\, u(0) = q $   is less or equal to $\bar r$, the  equation (\ref{eqqbase}) is linear, with solution
  $\, u(t;q)= q \cos(\omega t)$, $\omega = \sqrt{\alpha j^4 +2m}$. It follows that $\bar f_j'$ equals to $m$, and   the Hill equation (\ref{zjjbase}) reduces to a linear oscillator with constant angular frequency
$A=\sqrt{\beta j^2+ 2\gamma m}$. Thus it is stable and   $\Delta$ is simply the trace of the matrix $L$ with $\Delta   t=2 \pi/\omega$, \emph{ i.e. } $\Delta = 2 \cos(2\pi A/\omega)$.

In the case when $q> \bar r$, we must identify the intervals $I_i$  at whose end points $u(t;q) = \pm \bar r$.
We start  when  $j$ is even. Owing to (\ref{ftildaj}}) we have in the equation (\ref{eqqbase}),
$$  \bar f_j(r) =\frac{1}{2}( mr+m\bar r),\quad  \text{ if } \quad r>\bar r,\quad \bar f_j(r) =\frac{1}{2}( mr-m\bar r),\quad  \text{ if } \quad r<-\bar r,$$
therefore the  two angular  frequencies  of the Hill equation (\ref{zjjbase}) are
%\begin{equation}
%\label{Apari}
$$
A_0=\sqrt{\beta j^2+\gamma m}, \quad  \text{ if } \quad |u| \geq  \bar r, \qquad  A_1=\sqrt{\beta j^2+2\gamma m},  \quad  \text{ if } \quad |u| < \bar r.
$$
%\end{equation}

To determine the intervals in which the two cases occur, we note that $\bar f_j(r)$ is  an odd function, so that the potential for the equation (\ref{eqqbase}) is even.
 Therefore it is enough to compute $u(t;q)$ for  a quarter of a period, starting from $t=0$ to the first positive time $\, t_1 = t_1(q) $ when $\, u(t_1;q) = 0$, so that
 $\tau = 4 t_1$.

If we call $t_0= t_0(q) < t_1(q)$ the first positive time such that $u(t_0;q) = \bar r $,
then  the cycle $[-t_0, \tau-t_0]$ is the union of disjoint intervals $\cup_{i=0}^3 I_i$, such that
$$ |u| \geq \bar r, \quad t\in I_0\cup I_2; \qquad  |u| \leq  \bar r \quad t\in I_1\cup I_3. $$

Their  length are $\Delta t_i= 2t_0$,  for     $i =0,2$, and    $\Delta t_i=  2(t_1-t_0)$,  for     $i =1,3$. As a consequence, the potential of the equation (\ref{zjjbase}) is given by
%\begin{equation}
%\label{Q}
$$
 Q(t) = A^2_0 \mathbbm{1}_{I_0}(t) + A^2_1 \mathbbm{1}_{I_1}(t) + A^2_0 \mathbbm{1}_{I_2}(t) + A^2_1 \mathbbm{1}_{I_3}(t), \qquad t\in [-t_0, \tau -t_0],
 $$
% \end{equation}
and we have 4 transition matrices with $L_0 = L_2$, $L_1 = L_3$, so that
$$ M=(L_1 \, L_0)^2, \qquad \Delta(q)  ={\rm tr} M. $$

As for the computation of $t_0$, $t_1$, we define the 2 angular frequencies of the equation (\ref{eqqbase}),
$$ \omega_0=\sqrt{\alpha j^4 +m} , \qquad \omega_1 = \sqrt{\alpha j^4 +2m}, $$
so that  $u$ satisfies the equation $\, \ddot{u}+ \omega_0^2 u=-m\bar r\, $ on $I_0$, and $\ddot{u}+\omega_1^2 u=0$ on $I_1$.
On the first interval  $I_0 =  [-t_0, t_0[$, we obtain
\begin{equation*}
 u(t;q)=\left(q+\frac{m\bar r}{\omega_0^2} \right)  \cos\omega_0 t - \frac{m\bar r}{\omega_0^2}.
\end{equation*}

By solving the equation $\, u(t;q) = \bar r$, we get
%\begin{equation}
%\label{t0pari}
$$
t_0 = \frac{1}{\omega_0} {\rm acos}\,\left( \frac{\omega_0^2 \bar r+m\bar r}{\omega_0^2 q+m\bar r}\right).
$$
%\end{equation}

On the  second  interval $I_1$, $u$  solves  the equation  $\ddot{u}+\omega_1^2 u=0$, with
initial data $u(t_0;q)$ and $\dot{u}(t_0;q)$. Thus
$$u(t;q)=\bar r \cos(\omega_1 (t-t_0)) + \dot{u}(t_0;q)  \sin (\omega_1 (t-t_0))/\omega_1. $$

Since we have
$$ \dot{u}(t_0;q)= -\omega_0(q+\frac{m\bar r}{\omega_0^2})\sqrt{1-\cos^2 \omega_0 t_0}=-\sqrt{(q-\bar r)[\omega_0^2(q+\bar r)+2m\bar r]},$$
by solving the equation $\dot{u}(t;q)=0$, we get
%\begin{equation}
%\label{t1pari}
$$
t_1- t_0 = \frac{1}{\omega_1} {\rm atan}\,\left( \frac{\omega_1 \bar r}{B(q)}\right), \qquad B(q) =\sqrt{(q-\bar r)[\omega_0^2(q+\bar r)+2m\bar r]}.
$$
%\end{equation}

Now we come to the calculations in the case when  $j$ is odd. The function $\bar f_j(r)$    in (\ref{eqqbase}), for $|r|>\bar r$,   is given by
\begin{align*}
%\label{fdispari}
 &\bar f_j(r) =\frac{m}{2}(1+ 1/j)r+ \frac{m}{2}(1-i/j)\bar r,\quad r>\bar r, \\
 &\bar f_j(r) =\frac{m}{2}(1-1/j)r-\frac{m}{2} (1+1/j)\bar r,\quad r<-\bar r.
\end{align*}

The procedure is similar to the previous case, only a bit longer,  because  $\bar f_j(r)$ is no more odd, so that
  we need  to compute $u(t;q)$ on  a    half period.  We define $t_0$ the first positive time such that $u(t_0;q)= \bar r$,
and again we consider the cycle
$[-t_0, \tau-t_0]$ which is the union $\cup_{i=0}^3I_i$ of disjoint intervals,  such that
$$ u\geq \bar r, \quad t\in I_0; \quad |u| \leq \bar r, \quad t\in I_1\cup I_3; \quad  u \leq -\bar r, \quad t\in I_2.  $$

The Hill equation (\ref{zjjbase}) is now a four steps Meissner equation with potential
$$
 Q(t) = A^2_0 \mathbbm{1}_{I_0}(t) + A^2_1 \mathbbm{1}_{I_1}(t) + A^2_2 \mathbbm{1}_{I_2}(t) + A^2_1 \mathbbm{1}_{I_3}(t), \qquad t\in [-t_0, \tau -t_0],
$$
where
$$ A_0=\sqrt{\beta j^2+\gamma m(1+1/j)},\,\, A_1=A_3=\sqrt{\beta j^2+2\gamma m},\,\,  A_2=\sqrt{\beta j^2+\gamma m(1-1/j)},
$$
then  the  matrix $M$ becomes
$ M=L_1\,L_2\,L_1 \, L_0$.

 We define the 3 angular frequencies for the equation  (\ref{eqqbase}),
$$
\omega_0=\sqrt{\alpha j^4 +m(1+1/j)},\quad \omega_1=\sqrt{\alpha j^4 +2m},\quad \omega_2=\sqrt{\alpha j^4 +m(1-1/j)},
$$
and the constants
$$
D= m(1-1/j) \bar r,\quad E= m(1+1/j) \bar r,
$$
so that  $u$ satisfies the equation  $\ddot{u}+ \omega_0^2 u=-D$ on  $I_0$, the equation $\ddot{u}+  \omega_1^2 u=0$  on  $I_1\cup I_3$,  and   $\ddot{u}+ \omega_2^2 u=E$
on $I_2$.

Proceeding as above, for $t\in I_0$ we have,
\begin{equation*}
 u(t;q)=\left(q+\frac{D}{\omega_0^2} \right)  \cos\omega_0 t - \frac{D}{\omega_0^2},
\end{equation*}
thus, by finding $t_0$, we obtain half the length of $I_0$, that is
$$ \frac{\Delta t_0}{2} =  \frac{1}{\omega_0} {\rm acos}\,\left( \frac{\omega_0^2 \bar r+D}{\omega_0^2 q+D}\right).
$$

For $t\in I_1$, starting with the initial conditions at time $t_0$, we get the expression
\begin{equation*}
u(t;q)=\bar r \cos(\omega_1 (t-t_0))-\frac{B}{\omega_1} \sin (\omega_1 (t-t_0)), \quad B = -u(t_0;q)= \sqrt{(q-\bar r)[\omega_0^2(q+\bar r)+2D]}.
\end{equation*}

Thus, by finding the first time greater than $t_0$ such that  $u(t;q)=0$, we obtain half the length of $I_1$:
%\begin{equation}
%\label{t1dispari}
$$
\frac{\Delta t_1}{2} =   \frac{1}{\omega_1} {\rm atan}\,\left( \frac{\omega_1 \bar r}{B}\right).
$$
%\end{equation}

Finally, for $t\in I_2$,  we need  to solve the  third equation  $\ddot{u}+ \omega_2^2 u=E$ with  initial  data  at $t^* = t_0 + \Delta t_1$, that is
  $u(t^*;q)= -\bar r$,  $\dot{u}(t^*;q)=\dot{u}(t_0;q)=-B$.
Therefore, we have
\begin{equation*}
u(t;q)=-\left(\bar r+\frac{E}{\omega_2^2}\right)\cos \omega_2(t-t^*) -\frac{B}{\omega_2}\sin \omega_2 (t-t^*)+\frac{E}{\omega_2^2}.
\end{equation*}

By finding the first time greater than $t^*$ such that  $\dot{u}(t;q)=0$, we obtain  half the length of $I_2$:
$$
\frac{\Delta t_2}{2} =    \frac{1}{\omega_2} {\rm atan}\,\left( \frac{\omega_2B}{\omega_2^2 \bar r+E}\right).
$$

%%%%%%%%%%%%%%%%%

\end{document}